\documentclass[11pt]{article}
\usepackage[T1]{fontenc}
\usepackage{lmodern}
\usepackage{amssymb,amsmath,amsthm}
\usepackage{bbm}
\usepackage{graphicx}
\usepackage{color}
\usepackage[a4paper,margin=2.5cm]{geometry}
\usepackage{cases}
\newtheorem{thm}{Theorem}
\newtheorem{corr}[thm]{Corollary}
\newtheorem{lem}[thm]{Lemma}
\newtheorem{prop}[thm]{Proposition}
\def\erfc{\mathrm{erfc}}
\DeclareMathOperator{\var}{var}

\def\cbes#1#2#3{\xi^{(#1:#2\to#3)}}
\def\cbro#1#2#3{B^{(#1:#2\to#3)}}
\def\bes #1{\xi^{(#1)}}
\def\bro#1{B^{(#1)}}

\def\qop#1#2#3{q(#1,#2,#3)}

\def\tbes#1{\tilde \xi^{(#1)}}
\def\tcbes#1#2#3{\tilde\xi^{(#1:#2\to#3)}}

\def\R{\mathbb{R}}
\def\P{\mathbb{P}}

\def\E{\mathbb{E}}

\def\diffd{\mathrm{d}}
\newcommand{\indic}[1]{\mathbbm{1}\raisebox{-.4ex}{$\scriptstyle\{#1\}$}}

\title{Vanishing corrections for the position in a
linear model of FKPP fronts}

\author{Julien Berestycki, \'Eric Brunet,  Simon C.~Harris, Matthew I.~Roberts}
\date{\today}
\begin{document}
\maketitle

\begin{abstract}
Take the linearised FKPP equation
$
\partial_t h =\partial^2_x h +h
$
with boundary condition $h(m(t),t)=0$. Depending on the behaviour of the
initial condition $h_0(x)=h(x,0)$ we obtain the asymptotics ---~up to a
$o(1)$ term $r(t)$~--- of the absorbing boundary $m(t)$ such that $\omega(x) :=
\lim_{t\to\infty} h(x+m(t) ,t)$ exists and is non-trivial. In particular, as in
Bramson's results for the non-linear FKPP equation, we recover the
celebrated $-3/2\log t$ correction for initial conditions decaying faster
than $x^{\nu}e^{-x}$ for some $\nu<-2$.

Furthermore, when we are in this regime, the main result of the present
work is the identification (to first order) of the $r(t)$ term which
ensures the fastest convergence to $\omega(x)$. When $h_0(x)$ decays faster
than   $x^{\nu}e^{-x}$ for some $\nu<-3$, we show that $r(t)$ must be
chosen to be $-3\sqrt{\pi/  t}$ which is precisely the term predicted
heuristically by Ebert-van Saarloos \cite{Ebert2000} in the non-linear
case (see also \cite{Mueller2014,Brunet2015,Henderson2014}).  When the initial
condition decays as  $x^{\nu}e^{-x}$ for some $\nu\in[-3,-2)$, we show that
even though we are still in the regime where Bramson's correction is
$-3/2\log t$, the Ebert-van Saarloos correction has to be modified.

Similar results were recently obtained by Henderson \cite{Henderson2014} using
an analytical approach and only for compactly supported initial conditions.
\end{abstract}

\section{Introduction}

The celebrated Fisher-Kolmogorov-Petrovsky-Piscounof equation (FKPP)
in one dimension for $h:\R\times\R^+\to \R$ is:
\begin{equation}
\partial_t h = \partial_x^2 h + h -h^2,
\qquad h(x,0)=h_0(x).
\label{FKPP}
\end{equation}
This equation is a natural description of a reaction-diffusion
model \cite{Fisher1937, Kolmogorov1937, Aronson1978}. It is also related to branching Brownian motion: for the
Heaviside initial condition $h_0(x)=\indic{x<0}$, $h(x,t)$ is the
probability that the rightmost particle at time~$t$  in a branching
Brownian motion (BBM) is to the right of~$x$.

For suitable initial conditions where $h_0(x)\in[0,1]$, $h_0(x)$ goes to
1 fast enough as $x\to -\infty$ and $h_0(x)$ goes to 0 fast enough as $x\to\infty$, it is known that $h(x,t)$ develops into a travelling
wave: there exists a centring term $m(t)$ and an asymptotic shape
$\omega_v(x)$ such that
\begin{equation}
\lim_{t\to\infty} h\big(m(t)+x,t\big) = \omega_v(x)\in(0,1),
\label{limitTW}
\end{equation}
where $m(t)/t\to v$ and $\omega_v(x)$ is a travelling wave solution
to~\eqref{FKPP} with velocity~$v$: that is, the unique (up to
translation) non-trivial solution to 
\begin{equation}
\omega_v''+ v\,\omega_v'+\omega_v-\omega_v^2=0
\end{equation}
with $\omega_v(-\infty)=1$ and $\omega_v(+\infty)=0$.

In his seminal works \cite{Bramson83}, Bramson showed how the initial
condition $h_0$ (and in particular its large~$x$ asymptotic behaviour)
determines $m(t)$ in \eqref{limitTW}. For the important example
$h_0(x)=\indic{x<0}$ corresponding to the rightmost particle in BBM, he
finds
\begin{equation}
m(t)=2t-\frac32\log t + a +o(1)
\label{BramsonEx}
\end{equation}
for some constant~$a$, and a limiting travelling wave with (critical) speed
$v=2$. (Here and throughout, we use the notation $f(t)=o(1)$ to mean that
$f(t)\to 0$ as $t\to\infty$.)

What makes Bramson's results extremely interesting is their universality;
for instance Bramson proves \cite{Bramson83} that the previous result
still holds if the reaction term $h-h^2$ in \eqref{FKPP} is replaced by
$f(h)$ with $f(0)=f(1)=0$, $f'(0)=1$ and $f(x)\le x$. The universality goes
further than that,  and for many other front equations, it is believed and
sometimes known that the centring term~$m(t)$ follows the same kind
of behaviour as for~\eqref{FKPP}: one needs to compute a function
$v(\gamma)$ which has a minimum $v_c$ at a point $\gamma_c$ (in the FKPP
case \eqref{FKPP}, $v(\gamma)=\gamma+1/\gamma$, $\gamma_c=1$, $v_c=2$);
then for an initial condition decreasing like $e^{-\gamma x}$, the front
converges to a travelling wave with velocity $v(\gamma)$ if $\gamma\le
\gamma_c$ and critical velocity $v_c$ if $\gamma\ge\gamma_c$.

When the centring term $m(t)$ is defined as in \eqref{limitTW}, it is not
uniquely determined: if $m(t)$ is any suitable centring term, then $m(t)+o(1)$
is also a suitable centring term.
Instead one can try to give a more precise definition for $m(t)$. For example, one could reasonably ask for 
\begin{equation}
h\big(m(t),t\big)=\alpha\ \text{for some $\alpha\in(0,1)$}
\quad\text{or}\quad \partial_x^2h\big(m(t),t\big)=0
\quad\text{or}\quad m(t)=-\int\diffd x \, x \partial_x h(x,t)
\label{possiblemt}
\end{equation}
in addition to \eqref{limitTW}. In the case $h_0(x)=\indic{x<0}$, so that
$h(x,t)=\P(R_t>x)$ where $R_t$ is the position of the rightmost particle in
a BBM at time~$t$,
the first definition in~\eqref{possiblemt} would be the $\alpha$-quantile
of $R_t$, the second definition would be
the mode of the distribution of $R_t$, and the third definition would be the
expectation of $R_t$.

It has been heuristically argued
\cite{Ebert2000,Mueller2014,Henderson2014,Brunet2015}
that any quantity $m(t)$ defined as in \eqref{possiblemt}
behaves for large~$t$ as
\begin{equation}
m(t)=v_c t-\frac3{2\gamma_c}\log t +a -3\sqrt{\frac{2\pi}{\gamma_c^5
v''(\gamma_c)}}\times\frac1{\sqrt t}
+ o\Big(\frac1{\sqrt t}\Big),
\label{Ebert}
\end{equation}
for any front equation of the FKPP type and for any initial condition that
decays fast enough.
In the FKPP case \eqref{FKPP}, one has $\gamma_c=1$ and $v''(\gamma_c)=2$ so that $m(t)=2t-(3/2)\log t+a-3\sqrt{\pi/t} + o(1/\sqrt t)$.

Heuristically, the coefficient of the $1/\sqrt t$ term does not
depend on the precise definition of $m(t)$ because the front $h(x,t)$
converges very quickly to its limiting shape in the region where $h$ is
neither very close to 0 nor very close to 1, so that the difference between
any two reasonable definitions of $m(t)$ converges quickly (faster than $1/\sqrt t$) to some
constant. Note that the constant term ``$a$'' is expected to be
non-universal and to depend on the model, the initial condition and the
precise definition of $m(t)$.

As argued in \cite{Ebert2000}, the reason why the ``$\log t$'' and the ``$1/\sqrt t$'' terms in
\eqref{Ebert} are so universal is that they are driven by
the way the front develops very far on the right, in a region where it is
exponentially small and where understanding the position~$m(t)$
of the front is largely a matter of solving the linearised front equation.

However there is a catch: solving directly the linearised equation $\partial_t h=\partial_x^2 h +h$ with
(for instance) a step initial condition $h_0(x)=\indic{x<0}$, one finds
$h_\text{linear}(x,t)=\frac12 e^t \erfc(x/\sqrt{4t})$. Defining the
position $m(t)$ by $h_\text{linear}\big(m(t),t\big)=1$ gives 
$m(t)=2t-\frac12\log t+a+\mathcal O\big((\log^2 t)/t\big)$ rather
than~\eqref{BramsonEx}; the linearised
equation has the same velocity $2$ as for the FKPP equation, a logarithmic
correction but with a different prefactor and no $1/\sqrt t$ correction.
The problem is that with the linearised equation, the
$h_\text{linear}(x,t)$ increases exponentially on the left of $m(t)$ and
this ``mass'' pushes the front forward, leading to a $-\frac12\log t$ rather
than a $-\frac32\log t$ correction. This means that in order to recover the
behaviour of $m(t)$ for the FKPP equation, one must have a front
equation with some saturation mechanism on the left. The behaviour of $m(t)$
is not expected to depend on which saturation mechanism is chosen, but one
must be present.
For these reasons, we consider in this paper a linearised FKPP with
a boundary on the left, as in \cite{Henderson2014}.

We emphasize that, in the present work, the FKPP equation is only
a motivation: we do not attempt to establish the equivalence between the
FKPP equation and the linear model with a boundary. Our results are proved
only for the linear model with boundary, and we can only conjecture that
they do apply to the FKPP equation.

\section{Statement of the problem and main results}

We study the following linear partial differential equation with
initial condition $h_0(x)$ and a given boundary~$m:[0,\infty)\to\R$:
\begin{equation}
\begin{cases}
\partial_t h =\partial_x^2 h + h&\text{for $x>m(t)$},\\
h\big(m(t),t\big)=0,& h(x,0)=h_0(x).
\end{cases}
\label{eqboundary}
\end{equation}
Observe that without loss of generality we can (and will) insist that $m(0)=0$ since otherwise we can simply shift the reference frame by $m(0)$ by the change of coordinate $x \mapsto x-m(0)$.

The same system was studied in \cite{Henderson2014} by PDE methods for
compactly supported initial conditions. In this paper, we use probabilistic
methods, writing the solution of the heat equation as an expectation
involving Brownian motion with a killing boundary. We give more general
results, in particular lifting the compactly supported hypothesis.

If the boundary is linear, $m(t)=vt$, the problem is easily solved explicitly. However, as soon as $m(t)$
is no longer linear, gaining any explicit information about the solution is known to be hard (see for instance \cite{HerrmannTanre2015}) and there are few available results.

Motivated by the earlier FKPP discussion about convergence to a travelling
wave as in \eqref{limitTW}, we are looking for functions
$m:[0,\infty)\to\R$ and $\omega:[0,\infty)\to[0,\infty)$ such that
\begin{equation}
\lim_{t\to\infty} h\big(m(t)+x,t\big) =\omega(x) \,\,\text{ for all } x\geq 0
\label{limitTW2}
\end{equation}
with $\omega$ non-trivial, $\omega(0)=0$ and $\omega(x)>0$ for all $x>0$. Note that such a function $\omega$ necessarily satisfies
\begin{equation}
\omega''(x) + v\omega'(x) + \omega(x) = 0, \quad \forall x\ge 0.
\end{equation}
In this case, the boundary condition anchors the front. Requiring the
convergence of $h(m(t)+x,t)$ to a limiting shape means that $m(t)$ must
increase fast enough to prevent the mass near the front from growing
exponentially, but not so fast that it tends to zero. This provides
a saturation mechanism, and even though it might seem very unlike FKPP
fronts to have $h\big(m(t),t\big)=0$, as discussed earlier we do expect the
two systems to behave similarly.

\vspace{3mm}

\noindent
Throughout the article we use the following notation:
\vspace{-1.5mm}
\begin{itemize}
\item $f(x)\sim g(x)$ means $f(x)/g(x)\to1$ as $x\to\infty$;
\vspace{-2mm}
\item $f(x)=\mathcal O\big(g(x)\big)$ means there exists $C>0$ such that $|f(x)|\leq C |g(x)|$ for all large $x$;
\vspace{-2mm}
\item $f(x)=o\big(g(x)\big)$ means $f(x)/g(x)\to 0$ as $x\to\infty$.
\vspace{-2mm}
\item A random variable $G$ is said to have ``Gaussian tails'' if there
exist two positive constants $c_1, c_2$ such that $\P(|G|>z)\le c_1
\exp(-c_2 z^2)$ for all $z\ge0$.
\end{itemize}

\noindent
Our first theorem recovers the analogue of
Bramson's results for the system \eqref{eqboundary}, \eqref{limitTW2}.

\begin{thm}\label{thm1}
For each of the following bounded initial conditions $h_0$,
a twice continuously differentiable function $m(t)$ such that $m(0)=0$ and
 $m''(t)=\mathcal O(1/t^2)$ 
leads to a solution $h(x,t)$ to
\eqref{eqboundary} with a non-trivial limit~\eqref{limitTW2}
if and only if $m(t)$ has the following large time asymptotics where $a$ is
an arbitrary constant:

\begin{subequations}
\smallskip \noindent(a)\quad if $h_0(x) \sim A x^\nu e^{-\gamma x}$ with
$0<\gamma<1$ for large~$x$,
\begin{equation}
\begin{gathered}
m(t)=\Big(\gamma+\frac1\gamma\Big) t +\frac\nu\gamma\log t+ a + o(1),
\\\qquad\text{and then}\quad
\omega(x) = \alpha \Big(e^{-\gamma x}-e^{-\frac x\gamma}\Big)
\quad\text{with}\quad\alpha=
A e^{-\gamma a} \left(\frac1\gamma-\gamma\right)^\nu.
\end{gathered}
\label{thm1a}
\end{equation}

\smallskip \noindent(b)\quad if $h_0(x) \sim A x^\nu e^{- x}$ with
$\nu>-2$ for large~$x$,
\begin{equation}
\begin{gathered}
m(t)=2 t -\frac{1-\nu}2\log t +  a + o(1),
\\\qquad\text{and then}\quad
\omega(x) = \alpha x e^{- x}
\quad\text{with}\quad\alpha=\frac {Ae^{-a}}{\sqrt{\pi}}
2^\nu\Gamma\Big(1+\frac\nu2\Big) .
\end{gathered}
\label{thm1b}
\end{equation}

\smallskip \noindent(c)\quad if $h_0(x) \sim A x^{-2} e^{- x}$ for
large~$x$,
\begin{equation}
\begin{gathered}
m(t)=2 t -\frac32\log t +  \log\log t+ a + o(1),
\\\qquad\text{and then}\quad
\omega(x) = \alpha x e^{- x}
\quad\text{with}\quad\alpha=\frac {Ae^{-a}}{4\sqrt{\pi}}.
\end{gathered}
\label{thm1c}
\end{equation}

\smallskip \noindent(d)\quad if $h_0(x) ={\mathcal O}\big( x^\nu e^{-
x}\big)$ with $\nu<-2$ for large~$x$ and 
such that the value of $\alpha$ below is non-zero,
\begin{equation}
\begin{gathered}
m(t)=2 t -\frac32\log t +  a + o(1),
\\\text{and then}\quad
\omega(x) = \alpha x e^{- x}
\quad\text{with}\quad\alpha=
\frac{e^{-a-\Delta}}{2\sqrt\pi}
\int_{0}^\infty\diffd y\, h_0(y) y e^y
\psi_\infty(y),
\end{gathered}
\label{thm1d}
\end{equation}
where $\Delta$ and $\psi_\infty$ are quantities depending on the whole
function $m$ (and not only the asymptotics)
which are introduced (in \eqref{defDelta} and \eqref{limpsi}) in the
proofs.
\end{subequations}
\end{thm}
\textbf{Remarks.} 
\begin{itemize}
\item
From the probabilistic representation of $h(x,t)$ written later in the paper \eqref{q1},
it is clear that the solution $h(x,t)$ to \eqref{eqboundary}
must be an increasing function of $h_0$ and a decreasing function of
$m$ (in the sense that if $m^{(1)}(t)\ge m^{(2)}(t)$ for all $t$, then
$h^{(1)}(x,t)\le h^{(2)}(x,t)$ for all $x$ and $t$). This implies that the
$\alpha$ given in Theorem~\ref{thm1} must be increasing functions of $h_0$
and decreasing functions of $m$. This was obvious from the explicit
expression of $\alpha$ in cases (a), (b) and (c). In case (d), given the
complicated expressions for $\Delta$ and $\psi_\infty$, it is not obvious
at all from its expression that $\alpha$ decreases with $m$.
\item
Consider now a twice differentiable function $m$ without the assumption that
 $m''(t)=\mathcal O(1/t^2)$. The monotonicity of $h(x,t)$ with respect to
$m$ still holds, and by sandwiching such a $m$ between two sequences of
increasingly close functions that satisfy the $\mathcal O(1/t^2)$
condition, one can show easily in cases (a), (b) and (c) that
if $m$ has the correct asymptotics, then $h\big(m(t)+x,t\big)$
converges as in Theorem~\ref{thm1}. Case (d) is more difficult as both
$\Delta$ and $\psi_\infty$ might be ill defined when one does not assume
 $m''(t)=\mathcal O(1/t^2)$.
 \end{itemize}

\vspace{3mm}

We now turn to the analogue of the Ebert-van Saarloos correction 
\eqref{Ebert} for our model~\eqref{eqboundary}. As explained in the introduction and shown in Theorem \ref{thm1}, with a characterization as in
\eqref{limitTW2}, $m(t)$ is only determined up to $o(1)$. If we wish to improve upon Theorem \ref{thm1}, then we need a more precise definition for $m(t)$, analogous to
\eqref{possiblemt}. Natural
possible definitions could be
\begin{equation}
h\big(m(t)+1,t\big)=1\qquad\text{or}\qquad
\partial_x h\big(m(t),t\big)=1.
\label{possibmt}
\end{equation}
However, it is not obvious that such a function $m(t)$ even exists, would be unique or differentiable. We are furthermore interested only in the long time
asymptotics of $m(t)$. Therefore, instead of requiring something
like~\eqref{possibmt} we rather look, as in \cite{Henderson2014}, for the
function $m(t)$ such that the convergence \eqref{limitTW2} is as fast as
possible.

Our main result, Theorem \ref{thm2}, tells us how fast
$h\big(m(t)+x,t\big)$ converges for suitable choices of $m$ in case (d) of
Theorem \ref{thm1}. This case is the most classical as it contains, for
example, initial conditions with bounded support. It is the case studied by
Ebert-Van Saarloos and Henderson, and is the case for which universal
behaviour is expected. Theorem \ref{thm2} is followed by two corollaries
that highlight important consequences.

\begin{thm}\label{thm2}
Suppose that $h_0$ is a	bounded function
such that $h_0(x) = \mathcal O\big(x^\nu e^{-x}\big)$ for large~$x$
for some $\nu<-2$, and such that $\alpha$ defined in~\eqref{thm1d} is
non-zero. Suppose also that
$m$ is twice continuously differentiable with
\begin{equation}
m(t) = 2t -\frac32\log (t+1) + a + r(t)
\label{thm2:condm1}
\end{equation}
where $r(0)=-a$, $r(t)\to 0$  as $t\to\infty$ and $r''(t) = \mathcal
O(t^{-2-\eta})$ for large~$t$ for some $\eta>0$. Then for any $x\geq0$,
\begin{equation}
\label{thm2:main}
h\big(m(t)+x,t\big) = \alpha x e^{-x} \left[
1-r(t)-\frac{3\sqrt\pi}{\sqrt t}
+\mathcal O \left({t^{1+\frac\nu2}}\right)
+\mathcal O \left(\frac1{t^{\frac12+\eta}}\right)
+\mathcal O\left(\frac{\log t}{t}\right)
+\mathcal O\big(r(t)^2\big)
\right]
\end{equation}
with $\alpha$ as in \eqref{thm1d}.\\
If we further assume that $h_0(x)\sim A x^\nu e^{-x}$ for large~$x$ for
some $A>0$ and $-4<\nu<-2$, then
\begin{equation}
\label{thm2:main2}
h\big(m(t)+x,t\big) =  x e^{-x}\left[
\alpha\left(
1-r(t)-\frac{3\sqrt\pi}{\sqrt t}\right)
-b {t^{1+\frac\nu2}}
+ o \left({t^{1+\frac\nu2}}\right)
+\mathcal O \left(\frac1{t^{\frac12+\eta}}\right)
+\mathcal O\big(r(t)^2\big)
\right]
\end{equation}
with
\begin{equation}
b=-\frac{A}{\sqrt{4\pi}}e^{-a}2^{\nu+1}
\Gamma\Big(\frac\nu2+1\Big) >0.
\end{equation}
\end{thm}

This result allows us to bound the rate of convergence
$h\big(m(t)+x,t\big)$ to $\alpha x e^{-x}$: it is generically of order
$\max\big(1/\sqrt t,|r(t)|,t^{1+\nu/2}\big)$.

This also suggests that for $m(t)$ defined as in either choice of
\eqref{possibmt}, one should have $r(t)\sim -3\sqrt \pi/\sqrt t$ for
$\nu<-3$ and $r(t)\asymp t^{1+\nu/2}$ for $-3\le\nu<-2$. Note however that
we are not sure that such a $m(t)$ exists and, if it exists, we do not know
whether it satisfies the hypothesis on $m''(t)$ that we used in the
Theorem.

In the following two corollaries we highlight the best rates of convergence
of $h\big(m(t)+x,t\big)\to x e^{-x}$ that we can obtain from
Theorem~\ref{thm2}. For simplicity, we dropped the technical requirement
that $m(0)=0$ in the corollaries; the expression for $\alpha$ must
therefore be adapted.
\begin{corr}
Suppose that $h_0$ is a bounded function such that $h_0(x)=\mathcal
O\big(x^\nu e^{-x}\big)$ for large~$x$ with $\nu<-3$ and such that $\alpha$ 
is non-zero.
If we choose 
\begin{equation}
\label{magicr}
m(t)= 2t - \frac32 \log (t+1)  +a + \frac{c}{\sqrt{t+1}},
\end{equation}
then
\begin{align}
&\text{if $\nu\le-4$},&&
c = -3\sqrt\pi
\quad\Longleftrightarrow\quad
h\big(m(t)+x,t\big) = \alpha x e^{-x}+ \mathcal O\left(\frac{\log t}t\right),\\
&\text{if $-4<\nu<-3$},&&
c = -3\sqrt\pi
\quad\Longleftrightarrow\quad
h\big(m(t)+x,t\big) = \alpha x e^{-x}+ \mathcal O\left(t^{1+\frac\nu2}\right).
\label{imp2}
\end{align}
\end{corr}

\noindent
Note in particular that we have recovered the result of \cite{Henderson2014},
but with more general initial conditions (\cite{Henderson2014} only considered
compactly supported initial conditions).

\begin{corr}
Suppose that $h_0(x)$ is a bounded function such that
$h_0(x)\sim Ax^\nu e^{-x}$ for large~$x$ with $-4<\nu<-2$, with $m$, $r$ and
$b$ as in Theorem~\ref{thm2}. Then
\begin{align*}
&\text{if }-3<\nu<-2,&r(t)&=-\frac b\alpha t^{1+\frac\nu2}+
o \big(t^{1+\frac\nu2} \big)
&\Longleftrightarrow\ &h\big(m(t)+x,t\big) = \alpha x e^{-x}+
o \big(t^{1+\frac\nu2} \big),\\[2ex]
&\text{if }-4<\nu\leq-3,&r(t)&=-\frac{3\sqrt\pi}{\sqrt t} -\frac b\alpha t^{1+\frac\nu2}+ o\big(t^{1+\frac\nu2} \big)
&\Longleftrightarrow\ & h\big(m(t)+x,t\big) = \alpha x e^{-x}+ o\big(t^{1+\frac\nu2}\big).
\end{align*}
\end{corr}
Notice that for $h_0(x)\sim A x^{-3}e^{-x}$ the position $m(t)$ still
features a first order correction in $1/\sqrt t$ but with a coefficient
$-\big(3\sqrt\pi+\frac1{4\alpha}Ae^{-a}\big)$ which is different from the $\nu<-3$
case.

\section{Writing the solution as an expectation of a Bessel}

In this section, we write the solution to~\eqref{eqboundary} as an
expectation of a Bessel process.

We only consider functions $m(t)$ that are twice continuously
differentiable. For each given $m(t)$, \eqref{eqboundary} is
a linear problem. We first study the fundamental solutions $\qop{t}{x}{y}$ 
defined as
\begin{equation}
\begin{cases}
\partial_t q =\partial_x^2 q + q&\text{if $x>m(t)$},\\
\qop{t}{m(t)}{y} =0,& \qop{0}{x}{y} =\delta(x-y);
\end{cases}
\end{equation}
where $\delta$ is the Dirac distribution. Then
\begin{equation}
h(x,t)=\int_0^\infty \diffd y\,\qop{t}{x}{y} h_0(y).
\label{hq}
\end{equation}

It is clear that $e^{-t}\qop{t}{x}{y}$ is the solution to the heat equation with
boundary, and therefore
\begin{equation}
\qop{t}{x}{y}\diffd x=e^t \P\Big(\bro y_t\in\diffd x,\, \bro y_s>m(s)\, \forall
s\in(0,t)\Big),
\label{q1}
\end{equation}
where $t\mapsto \bro{y}_t$  is  the Brownian motion started from
$\bro y_0=y$ with the normalization
\begin{equation}
\E\big[(\bro y_{s+h}-\bro y_s)^2\big]=2h.
\label{normalize}
\end{equation}

Suppose $f:[0,\infty)\to\R$ is a continuous function, and $A_t(f)$ is a measurable functional that depends only on $f(s),\,s\in[0,t]$. Then by Girsanov's theorem,
\begin{equation}
\E\big[A_t(\bro y)\big]=e^{-\frac14\int_0^t\diffd s\, m'(s)^2}
\E\Big[A_t( m+\bro y)\,e^{-\frac12\int_0^t  m'(s)\,\diffd\bro y_s}\Big].
\end{equation}
Plugging into \eqref{q1} at position $m(t)+x$ instead of $x$, we get
\begin{equation}
\qop{t}{m(t)+x}{y}\diffd x = e^{t-\frac14\int_0^t\diffd s\,m'(s)^2}
\E\Big[\indic{\bro y_t\in\diffd x}\indic{\bro y_s>0\, \forall s\in(0,t)}e^{-\frac12\int_0^t m'(s)\,\diffd\bro y_s}\Big].
\label{q2}
\end{equation}
We recall that, by the reflection principle, the probability that
a Brownian path started from $y$ stays positive and ends in $\diffd x$ is:
\begin{equation}
\P\big(\bro y _t\in\diffd x,\,\, \bro y _s>0\, \forall s\in(0,t)\big) =\frac1{\sqrt{\pi t}}{\sinh\big(\frac{xy}{2t}\big)}e^{-\frac{x^2+y^2}{4t}}\diffd x.
\label{positive}
\end{equation}
Using \eqref{positive}, we write \eqref{q2} as a conditional expectation:
\begin{equation}
\qop{t}{m(t)+x}{y}=\frac{\sinh\big(\frac{xy}{2t}\big)}{\sqrt{\pi t}} 
e^{-\frac{x^2+y^2}{4t}+t-\frac14\int_0^t\diffd s\,m'(s)^2}
\E\Big[e^{-\frac12\int_0^t m'(s)\,\diffd\cbes tyx_s}\Big],
\label{q3}
\end{equation}
where $\cbes tyx_s,\, s\in[0,t]$ is a Brownian motion (normalized as
in~\eqref{normalize}) started
from $y$ and conditioned not to hit zero for any $s\in(0,t)$ and to be at
$x$ at time $t$. Such a process is called a Bessel-3 bridge, and
we recall some properties of Bessel processes and bridges in
Section~\ref{sec:toolbox}.

It is convenient to think of the path $s\mapsto\cbes tyx _s$ as the
straight line $s\mapsto y+(x-y)s/t$ plus some fluctuations. This leads us
to define
\begin{equation}
\begin{aligned}
\psi_t(y,x)&:=\E\Big[e^{-\frac12\int_0^t
m'(s)\,\big(\diffd\cbes tyx_s-\frac{x-y}t\diffd s\big)}\Big]
=\E\Big[e^{-\frac12\int_0^t m'(s)\,\diffd\cbes tyx_s}\Big]
e^{\frac{m(t)}{2t}(x-y)}
,
\\&\hphantom{:}=
\E\Big[e^{\frac12\int_0^t
m''(s)\,\big(\cbes tyx_s-(y+\frac{x-y}ts)\big)\diffd s)}\Big],
\end{aligned}
\label{defpsi}
\end{equation}
where we have used integration by parts. With this quantity, \eqref{q3} now reads
\begin{equation}
\qop{t}{m(t)+x}{y}=\frac{\sinh\big(\frac{xy}{2t}\big)}{\sqrt{\pi t}} 
e^{\frac{m(t)}{2t}(y-x)-\frac{x^2+y^2}{4t}+t-\frac14\int_0^t\diffd s\,m'(s)^2}
\psi_t(y,x),
\label{q4}
\end{equation}
and the main part of the present work is to estimate $\psi_t(y,x)$.

\section{The Bessel toolbox}
\label{sec:toolbox}

Before we begin our main task, we need some fairly standard estimates on
Bessel-3 processes and Bessel-3 bridges. From here on, we refer to these
simply as Bessel processes and Bessel bridges; the ``3'' will be implicit.
We include proofs for completeness.

We build most of our processes on the same probability space. We fix a driving
Brownian motion $(B_s,s\geq0)$ started from $0$ under a probability measure
$\P$, with the normalization $\E[B_t^2]=2t$.

For each $y\ge0$ we introduce a Bessel process $\bes y$ started from $y$ as the strong
solution to the SDE
\begin{equation}\label{sde bessel}
\bes y_0 = y, \qquad
\diffd \bes  y  _s = \diffd B_s + \frac 2 {\bes
y _s}\diffd s.
\end{equation}
It is well-known that $\bes y_s$ has the law of a Brownian motion conditioned to never hit zero.

We also introduce,  for each $t\ge0$ and $y\ge0$
\begin{equation}
\cbes t y 0_s
 = \frac{t-s}t\bes y_{\frac{st}{t-s}}\quad\text{for $s\in[0,t)$}.
\label{timechange}
\end{equation}
This process is a Bessel bridge from $y$ to 0 in time $t$, which is
a Brownian motion started from $y$ and conditioned to hit 0 for
the first time at time $t$.
One can check by direct substitution that $\cbes t y 0 _s$ solves
\begin{equation}
\cbes ty0_0=y,\qquad
\diffd \cbes ty0_s = \diffd \tilde B_{t,s} +\left(\frac 2{\cbes ty0_s}-\frac
{\cbes ty0_s}{t-s}\right)\diffd s,
\label{introcbes0}
\end{equation}
where for each $t$, 
$(\tilde B_{t,s},\, s\in[0,t))$ is the strong solution to
\begin{equation}\label{Btsdef}
\tilde B_{t,0}=0,\qquad
\diffd \tilde B_{t,s} =\frac {t-s}t\diffd \Big(B_{\frac{ts}{t-s}}\Big),
\end{equation}
and is thus itself a Brownian motion.

One can compute directly the law of the Brownian motion conditioned to
hit zero for the first time at time $t$ using \eqref{positive} and check
that this law solves the forward Kolmogorov equation (or Fokker
Planck equation) associated with the SDE (or Langevin equation)
\eqref{introcbes0}.

Similarly, we construct the Bessel bridge from $y$ to $x$ in time $t$,
the Brownian motion conditioned not to hit zero for any $s\in(0,t)$ and to be at
$x$ at time $t$, through
\begin{equation}
\cbes tyx_0=y,\qquad
\diffd \cbes tyx_s = \diffd \tilde B_{t,s} +\bigg(\frac x{t-s}\coth\frac{x
\cbes tyx_s}{2(t-s)}-\frac {\cbes tyx_s}{t-s}\bigg)\diffd s.
\label{introcbes}
\end{equation}

The advantages of constructing all the processes from a single Brownian
path $s\mapsto B_s$ is that they can be compared directly, realization by
realization. In particular we use the following comparisons:
\begin{lem}\label{beslem}
For any $y\ge z\ge0$ and $s\ge0$,
\begin{equation}
\bes z_s \le \bes y_s\le \bes z _s+ y-z\quad \text{ and } \quad y+B_s\le\bes y_s.
\label{beslem1}
\end{equation}

Furthermore, for any $y\ge0$, $x\geq z\geq0$, $t\ge0$ and $s\in[0,t]$,
\begin{equation}
\cbes t00_s \le \cbes ty0_s\le \cbes t00 _s+y\frac {t-s}t
,
\qquad
\cbes tyz_s \le \cbes tyx_s\le \cbes tyz _s+\frac {(x-z)s}t
.
\label{beslem2}
\end{equation}
\end{lem}

\begin{proof}
To prove \eqref{beslem1} we make three observations.
\begin{itemize}
\item[--] The processes $\bes y_s$ and $y +B_s$ both start from $y$ and 
\begin{equation}
\diffd \big(\bes y_s-(y +B_s)\big) = \frac {\diffd s}{\bes y_s} >0, \quad
s>0,
\end{equation}
so that $\bes y_s > y+B_s$ for all $s>0$ and  $y\ge 0$.
\item[--] $\bes y_s$ and $\bes z_s$ follow the same SDE \eqref{sde bessel}
and $\bes y_0\ge\bes z_0$, so the two processes must remain ordered at all
times (see for instance \cite{kunita1997}).
\item[--] We have 
\begin{equation}
\diffd (\bes y_s -\bes z_s) = \left(\frac1{\bes y_s }- \frac1{\bes
z_s}\right)\diffd s,
\end{equation}
and since $\bes y_s  \geq \bes z_s$ for all $s\ge 0$ we see that $\bes y_s
-\bes z_s$ is decreasing, yielding $\bes y_s - \bes z_s\leq y-z$ for all
$s\geq 0$.
\end{itemize}

The inequalities in the left part of \eqref{beslem2} are a direct
consequence of \eqref{beslem1} through the change of
time~\eqref{timechange}. We now focus on the inequalities in the right part
of \eqref{beslem2}. First we assume that $z>0$.

The fact that for $x\geq z$ we have $\cbes tyx_s \ge \cbes tyz_s$ follows from the fact that $x\coth(ax)\ge z\coth(az)$ for any $a>0$ and $x\ge z$.

For the other inequality,  the fact that $u(\coth u-1)$ is decreasing
yields that 
\begin{equation}\begin{aligned}
\diffd\cbes tyx_s&=\diffd \tilde B_{t,s} 
+\frac2{\cbes tyx_s}\times\frac{x\cbes
tyx_s}{2(t-s)}\bigg(\coth\frac{x\cbes tyx_s}{2(t-s)}-1\bigg)\diffd
s +\frac{x-\cbes tyx_s}{t-s}\diffd s
\\&\le
\diffd \tilde B_{t,s} 
+\frac2{\cbes tyx_s}\times\frac{z\cbes
tyz_s}{2(t-s)}\bigg(\coth\frac{z\cbes tyz_s}{2(t-s)}-1\bigg)\diffd
s +\frac{x-\cbes tyx_s}{t-s}\diffd s
\\&\le
\diffd \tilde B_{t,s} +\frac z
{t-s}\bigg(\coth\frac{z\cbes tyz_s}{2(t-s)}-1\bigg)\diffd
s +\frac{x-\cbes tyx_s}{t-s}\diffd s,
\end{aligned} \end{equation}
so that, writing $\zeta_s :=\cbes tyx_s - \cbes tyz_s\ge0$ for
the difference process,
\begin{equation}
\diffd \zeta_s 
\le \frac{x-z-\zeta_s}{t-s}\diffd s.
 \end{equation}
But the solution to $\frac{\diffd\phi_s}{\diffd s} =(x-z-\phi_s)/(t-s)$ and
$\phi_0=0$ is $\phi_s=(x-z)s/t$, implying that
$\zeta_s\le (x-z)s/t$, which concludes the proof for $z>0$. For the case
$z=0$ the proof is the same but uses the inequalities $1 \leq u\coth u \leq
1 + u$ for $u\geq0$.
\end{proof}

We note that, intuitively, as the length of a Bessel bridge tends to
infinity, on any compact time interval the bridge looks more and more like
a Bessel process. Similarly, as the start point of a Bessel process tends to
infinity, on any compact interval it looks more and more like a Brownian
motion relative to its start position. We make this precise in the lemma
below.

\begin{lem}\label{bessconv}
For all $s\ge 0$ and $y\ge 0$,
\begin{equation}\label{bessconv1}
\cbes t y 0_s \to \bes y_s \,\,\,\, \text{as $t\to\infty$}.
\end{equation}
For all $s\ge0$
\begin{equation}\label{bessconv2}
\bes y_s - y \to B_s \,\,\,\, \text{as $y\to\infty$}.
\end{equation}
For all $s\ge 0$ and any $y_t\to\infty$ as $t\to\infty$,
\begin{equation}\label{bessconv3}
\cbes t {y_t} 0_s - y_t\frac{t-s}{t} \to B_s \,\,\,\, \text{ as $t\to\infty$}.
\end{equation}
\end{lem}

\begin{proof}
For \eqref{bessconv1}, we simply recall \eqref{timechange} which defined 
\begin{equation}\cbes t y 0_s = \frac{t-s}t\bes y_{\frac{st}{t-s}}\quad\text{for $s\in[0,t)$},\end{equation}
so we are done by continuity of paths.

For \eqref{bessconv2}, recall from Lemma \ref{beslem} that $\bes y_s - y\geq B_s$. This both gives us the required lower bound, and tells us that for any $s\ge 0$, $\inf_{u\in[0,s]} \bes y_u \to\infty$ as $y\to\infty$. Thus
\begin{equation}\bes y_s - y = B_s + 2\int_0^s \frac{1}{\bes y_u}\diffd u \leq  B_s + \frac{2s}{\inf_{u\in [0,s]}\bes y_u} \to  B_s \text{ as } y\to\infty.\end{equation}

Finally, for \eqref{bessconv3}, we write
\begin{equation}\label{3parts}
\cbes t {y_t} 0_s - y_t\Big(\frac{t-s}{t}\Big) = \bigg[\bes {y_t}_s -
y_t\bigg]
-
\bigg[\frac{s}{t} (\bes{y_t}_s - y_t)\bigg] +
\bigg[
\Big(\frac{t-s}{t}\Big)(\bes{y_t}_{s+\frac{s^2}{t-s}} - \bes{y_t}_s)
\bigg].
\end{equation}
By \eqref{bessconv1}, $\bes{y_t}_s-y_t \to B_s$. By \eqref{beslem1}, $B_s \leq \bes{y_t}_s - y_t \leq \bes 0 _s$, so
\begin{equation}\frac{s}{t} (\bes{y_t}_s - y_t) \to 0 \quad\text{ as } t\to\infty.\end{equation}
Using our coupling between the Bessel processes and Brownian motion we have $\diffd B_u \le \diffd \bes {y_t}_u \le \diffd \bes 0_u$ for all $u\ge 0$ and hence
\begin{equation} B_{s+\frac{s^2}{t-s}} - B_s \leq \bes{y_t}_{s+\frac{s^2}{t-s}} - \bes{y_t}_s \leq \bes 0 _{s+\frac{s^2}{t-s}} - \bes 0_s\end{equation}
so by continuity of paths,
\begin{equation}\Big(\frac{t-s}{t}\Big)(\bes{y_t}_{s+\frac{s^2}{t-s}} -
\bes{y_t}_s) \to 0 \quad\text{ as } t\to\infty,
\end{equation}
which concludes the proof of \eqref{bessconv3}.
\end{proof}

We need the fact that the increments of a Bessel process over time $s$ 
are roughly of order $s^{1/2}$. By paying a small price on the exponent, we
obtain the following uniform bounds:
\begin{lem}\label{bessfluc}
For any $\epsilon>0$ small enough,
there exists a positive random variable $G$ with Gaussian tail such that
uniformly in $s\ge0$ and $y\ge0$,
\begin{equation}
\big| \bes y _{s} - y \big|
   \le G \max\Big(s^{\frac12-\epsilon},s^{\frac12+\epsilon}\Big)\quad \text{ and }
\quad
\big|B_s\big|      \le
G \max\Big(s^{\frac12-\epsilon},s^{\frac12+\epsilon}\Big).
\label{bessfluc:1}
\end{equation}
Furthermore, uniformly in $x\ge0$, $y\ge0$, $t\ge0$  and $0\le s\le t$,
\begin{equation}
\left| \cbes t y x_s-\Big(y+\frac{x-y}t s\Big)\right|\le
G \max\Big(s^{\frac12-\epsilon},s^{\frac12+\epsilon}\Big).
\label{bessfluc:2}
\end{equation}
\end{lem}
\begin{proof}
From \eqref{beslem1} we have $B_s\le\bes y_s-y\le\bes 0_s$. Also by
symmetry $\P(|B_s|>x)= 2\P(B_s>x)$. Thus to prove \eqref{bessfluc:1}, it
is
sufficient to show that 
\begin{equation}\P\left(\sup_{s>0}\frac{\bes 0_s}{\max(s^{1/2-\epsilon},s^{1/2+\epsilon})} > x \right) \leq c_1 e^{-c_2 x^2}\end{equation}
for some positive $c_1$ and $c_2$.
The proof is is elementary and we defer it to an appendix.

To prove~\eqref{bessfluc:2}, notice that from \eqref{beslem2} we have
\begin{equation}
\cbes t y x_s-\Big(y+\frac{x-y}t s\Big)\le \cbes t00_s.
\end{equation}
But from the change of time~\eqref{timechange} and \eqref{bessfluc:1},
\begin{equation}
\label{boundcbes}
\cbes t00_s=\frac{t-s}t\bes 0_{\frac{st}{t-s}}
\le
 G \frac{t-s}t\max\left\{\Big(\frac{st}{t-s}\Big)^{\frac12-\epsilon},\Big(
\frac{st}{t-s}\Big)^{\frac12+\epsilon}\right\}
\le 
G \max\Big(s^{\frac12-\epsilon},s^{\frac12+\epsilon}\Big),
\end{equation}
where the last step is obtained by pushing the $(t-s)/t$ inside the $\max$.
This provides the upper bound of \eqref{bessfluc:2}. For the lower bound,
we introduce Brownian bridges $s\mapsto \cbro tyx_s$ started from
$y$ and conditioned to be at $x$ at time $t$. We couple the Brownian
bridge to the Bessel bridges by building them over the family $\tilde
B_{t,s}$ of Brownian motions defined in \eqref{Btsdef}:
\begin{equation}
\cbro tyx_0=y,
\qquad
\diffd \cbro tyx_s = \diffd \tilde B_{t,s}
+\frac {x-\cbro tyx_s}{t-s}\diffd s.
\label{defcbro}
\end{equation}
One can check directly that
\begin{equation}\cbro tyx_s=y+\frac{x-y}{t}s+\cbro t00_s.\end{equation}
Furthermore,
by comparing~\eqref{defcbro} to \eqref{introcbes}, it is immediate from the fact that
$\coth u\geq 1$ for all $u\geq0$ that $\cbes tyx_s\ge\cbro tyx_s$. Therefore
\begin{equation}
\cbes t y x_s-\Big(y+\frac{x-y}t s\Big)\ge \cbro t00_s.
\end{equation}
Also, as in \eqref{timechange}, we can relate $B_s$ and $\cbro
t00_s$ through a time change:
\begin{equation}
\cbro t00_s=\frac{t-s}{t}B_{\frac{st}{t-s}}\quad\text{for $s\in[0,t)$},
\end{equation}
and, as in \eqref{boundcbes},
\begin{equation}
\big|\cbro t00_s\big|=\frac{t-s}t\Big|B_{\frac{st}{t-s}}\Big|
\le
 G \frac{t-s}t\max\left\{\Big(\frac{st}{t-s}\Big)^{\frac12-\epsilon},\Big(
\frac{st}{t-s}\Big)^{\frac12+\epsilon}\right\}
\le 
G \max\Big(s^{\frac12-\epsilon},s^{\frac12+\epsilon}\Big),
\end{equation}
which concludes the proof.
\end{proof}

\section{Simple properties of $\psi_t(y,x)$ and proof of Theorem~\ref{thm1}}
As in the hypothesis of Theorem~\ref{thm1}, we assume throughout this
section that $m$ is twice continuously differentiable with
\begin{equation}
m(0)=0
\quad \text{ and } \quad m''(s)=\mathcal O\Big(\frac1{s^2}\Big).
\label{mprop1}
\end{equation}
The large~$s$ behaviour of $m''(s)$ implies that there exists a $v$ such
that, for large~$s$,
\begin{equation}
m'(s)=v+\mathcal O\Big(\frac1s\Big) \quad \text{ and } \quad
m(s) = v s + \mathcal O(\log s).
\label{mprop2}
\end{equation}
We define
\begin{equation}
\Delta=\frac14\int_0^\infty\diffd s\, (m'(s)-v)^2,
\label{defDelta}
\end{equation}
which is finite because of \eqref{mprop1}.

\subsection{Simple properties of $\psi_t(y,x)$}

We recall from \eqref{defpsi} that the main quantity we are interested in is
\begin{equation}
\psi_t(y,x)=\E\big[e^{I_t(y,x)} \big],
\label{psiI}
\end{equation}
with
\begin{equation}
\label{defIt}
I_t(y,x)=\frac12 \int_0^t\diffd s\, m''(s)\Big(\cbes tyx_s-\Big(y+\frac{x-y}ts\Big)\Big),
\end{equation}
where we recall that $\cbes t y x_s,\, s\in[0,t]$ is a Bessel bridge from
$y$ to $x$ over time $t$. We mainly need to consider $x=0$ so
we use the shorthand
\begin{equation}
\psi_t(y):=\psi_t(y,0).
\end{equation}
We also define
\begin{equation}\label{defI}
I(y) =\frac12  \int_0^\infty\diffd s\,m''(s)\big(\bes y_s-y\big)
\end{equation}
where $\bes y_s,\, s\geq0$ is a Bessel process started from $y$.

\begin{prop}\label{easybounds}
The function $\psi_t(y,x)$ has the following
properties:\begin{itemize}
\item It is bounded away from zero and infinity: there exist two positive
constants $0<K_1<K_2$ depending on the function $m''(s)$ such that for
any $x$, $y$, $t$,
\begin{equation}
K_1  \le \psi_t(y,x)\le K_2.
\end{equation}
\item It hardly depends on $x$ for large times: recalling that $\psi_t(y):=\psi_t(y,0)$,
\begin{equation}
\psi_t(y,x)=\psi_t(y)\Big(1+x\,\mathcal O\Big(\frac{\log
t}t\Big)\Big)\quad\text{uniformly in $y$ and $x$}.
\label{noxdep}
\end{equation}
\item For fixed $y$, it has a finite and positive limit as $t\to\infty$:
\begin{equation}
\psi_\infty(y):=
\lim_{t\to\infty}\psi_t(y)=\E\Big[e^{I(y)}\Big]>0.
\label{limpsi}
\end{equation}
\item The large time limit~$\psi_\infty(y)$ has a well-behaved large~$y$
limit: for any function $t\mapsto y_t$ that goes to infinity as $t\to\infty$,
\begin{equation}
\lim_{y\to\infty}\psi_\infty(y)=\lim_{t\to\infty}\psi_t(y_t)
= 
\E\Big[e^{\frac12\int_0^\infty\diffd s\,m''(s) B_s}\Big] =
e^{\Delta}.
\end{equation}
\end{itemize}
\end{prop}

\begin{proof}
For the first result, Lemma~\ref{bessfluc} tells us that
\begin{equation}
\left| \cbes t y x_s-\Big(y+\frac{x-y}t s\Big)\right|\le G
\max\Big(s^{\frac12-\epsilon},s^{\frac12+\epsilon}\Big),
\label{unib1}
\end{equation}
where $G>0$ is a random variable with Gaussian tail independent of
$t$, $y$ and $x$. Then, since $m''(s) = \mathcal O(1/s^2)$,
\begin{equation}
\Big|I_t(y,x)\Big|
\le\frac12\int_0^\infty \diffd s\, \big|m''(s) \big| G \max\Big(s^{\frac12-\epsilon},s^{\frac12+\epsilon} \Big)
= G \mathcal O(1).
\label{unib2}
\end{equation}

For the second result, we compare paths going to $x$ with paths going to 0: we know from Lemma~\ref{beslem} that $0\le\cbes t y 0_s
-\cbes t y x_s
+ x  s/t\le xs/t$, so
\begin{equation}
\begin{aligned}
\big|I_t(y,0)-I_t(y,x)\big|
&\le\frac12\int_0^{t}\diffd s\,\big|m''(s)\big|\times\Big|\cbes t y 0_s
-\cbes t y x_s
+\frac x t s\Big|
\\&\le\frac x {2t} \int_0^{t}\diffd s\,\big|m''(s)\big|
s = x \mathcal O\Big(\frac{\log t}t\Big).
\end{aligned}
\end{equation}

We now turn to the third result.
For any fixed $s$ and $y$, Lemma \ref{bessconv} tells us that $\cbes t y 0_s\to \bes y_s$ as $t\to\infty$.
Thus, using \eqref{unib1} and \eqref{unib2}, we can apply dominated convergence and obtain
\begin{equation}
I_t(y,0) = \frac12\int_0^t\diffd s\, m''(s)\Big(\cbes ty0_s- y \frac{t-s}{t}\Big) \to \frac12\int_0^\infty \diffd s\, m''(s) \big(\bes y_s-y\big) = I(y).
\end{equation}
Furthermore, as the bound \eqref{unib2} is a random variable with Gaussian tails, using dominated convergence again we get
\begin{equation}
\lim_{t\to\infty}\E\big[e^{ I_t(y,0)}\big]
= \E\big[e^{ I(y)}\big].
\label{limitI}
\end{equation}

For the fourth statement, by Lemma \ref{bessconv} for any fixed $s$ we have
\begin{equation}
\lim_{y\to\infty}\big(\bes {y}_s-y\big)= B_s\quad\text{and}\quad
\lim_{t\to\infty}\Big(\cbes t {y_t}0_s -y_t\frac{t-s}t\Big) = B_s.
\label{makealem}
\end{equation}
Then, by dominated convergence using again a uniform Gaussian bound from Lemma \ref{bessfluc},
\begin{equation}\label{dcgb}
\lim_{y\to\infty}\psi_\infty(y)=\lim_{t\to\infty}\psi_t(y_t)
= 
\E\Big[e^{\frac12\int_0^\infty\diffd s\,m''(s) B_s}\Big].
\end{equation}
It now remains to compute the right-hand-side. Let 
\begin{equation}
X_t:=\frac12\int_0^t\diffd s\,m''(s)B_s.
\end{equation}
By integration by parts,
\begin{equation}
X_t = \frac12 m'(t)B_t-\frac12\int_0^t m'(s)\,\diffd B_s = \frac12\int_0^t
\big(m'(t)-m'(s)\big)\diffd B_s
\end{equation}
so $X_t$ is a time change of Brownian motion with
\begin{equation}
\E\Big[e^{X_t}\Big] = e^{\frac12 \var(X_t)} = e^{\frac18
\int_0^t(m'(t)-m'(s))^2\, 2\diffd s }
\to e^{\frac14 \int_0^\infty (v-m'(s))^2\,\diffd s}
= e^{\Delta}.
\end{equation}
Therefore, by dominated convergence as in \eqref{dcgb}, $\E[e^{X_\infty}] = e^\Delta$.
\end{proof}

\bigbreak

\subsection{Proof of Theorem~\ref{thm1}}

Since $m(0)=0$ and $m''(s)=\mathcal O(1/s^2)$, we can write  
$m(s)= v s +\delta(s)$ with $\delta(0)=0$, $\delta(s)=\mathcal O(\log s)$, and $\delta'(s)=\mathcal O(1/s)$. 
Note that
\begin{equation}
\int_0^t\diffd s\, m'(s)^2=\int_0^t\diffd
s\Big(v^2+2v\delta'(s)+\delta'(s)^2\Big)=v^2t+2v\delta(t)
+4\Delta+\mathcal O\Big(\frac1t\Big),
\end{equation}
where we recall that $\Delta=\frac14\int_0^\infty\diffd s\,\delta'(s)^2$.
We now fix $x>0$, so that any terms written as $\mathcal O(f(t))$ might depend on $x$; since
$x$ is fixed this will not matter. For instance, instead of
\eqref{noxdep} we simply write that
 $\psi_t(y,x) = \psi_t(y)e^{\mathcal O(\frac{\log
 t}{t})}$.

We recall \eqref{q4}:
\begin{equation}\qop{t}{m(t)+x}{y}=\frac{\sinh\big(\frac{xy}{2t}\big)}{\sqrt{\pi t}} e^{\frac{m(t)}{2t}(y-x)-\frac{x^2+y^2}{4t}+t-\frac14\int_0^t\diffd s\,m'(s)^2} \psi_t(y,x).
\end{equation}
Substituting in the estimate above we get
\begin{equation}
\qop{t}{m(t)+x}{y}=\frac{1}{\sqrt{\pi t}} 
e^{t\big(1-\frac{v^2}4\big) -\frac
v 2\delta(t)-\Delta-\frac v2x + \mathcal O\big(\frac{\log t}t\big)}\sinh\big(\frac{xy}{2t}\big)
e^{\frac{v}{2}y + \frac{\delta(t)}{2t}y}\psi_t(y)e^{-\frac{y^2}{4t}}.
\end{equation}
Then since $h(x,t) = \int_0^\infty \diffd y\,\qop{t}{x}{y} h_0(y)$---see
\eqref{hq}---we have
\begin{equation}
h\big(m(t)+x,t\big) = \frac{1}{\sqrt{4\pi }t^{3/2}} 
e^{t\big(1-\frac{v^2}4\big) -\frac
v 2\delta(t)-\Delta-\frac v2x+\mathcal O\big(\frac{\log t}t\big)} H(x,t),
\label{hH}
\end{equation}
with
\begin{equation}
H(x,t)=\int_0^\infty\diffd y\,h_0(y)2t\sinh\left(\frac{xy}{2t}\right)
e^{\frac v 2 y +\frac{\delta(t)}{2t}y}\psi_t(y)e^{-\frac{y^2}{4t}}.
\label{defH}
\end{equation}
We now must choose $v$ and $\delta(t)$, depending on the initial condition,
such that~\eqref{hH} has a finite and non-zero limit as $t\to\infty$.

We use the following simple calculus lemma to evaluate $H(x,t)$. We defer
the proof to the end of this section.

\begin{lem}\label{lem:calculus}
Let $\phi(y)$ a bounded function such that
\begin{equation}
\phi(y)\sim A y^\alpha \qquad\text{as~$y\to\infty$}
\label{calculus1}
\end{equation}
for some $A>0$ and some $\alpha$.
If $\epsilon_t=o\big(t^{-1/2}\big)$ then, as $t\to\infty$,
\begin{subnumcases}{\int_0^\infty\diffd y\, \phi(y) e^{-\frac{y^2}{4t}+\epsilon_t y}\psi_t(y)}
\sim A\,2^\alpha e^{\Delta}\Gamma\Big(\frac{1+\alpha}2\Big) t^{\frac{1+\alpha}2} & if \quad $\alpha>-1$ \label{calculus2a}\\ [2ex]
\sim \frac A2 e^{\Delta}\log t												 	 & if \quad $\alpha=-1$ \label{calculus2b}\\[2ex]
\to \int_0^\infty\diffd y\,\phi(y)\psi_\infty(y)								 & if \quad $\alpha<-1$.\label{calculus2c}
\end{subnumcases}
If \eqref{calculus1} is replaced by $\phi(y)=\mathcal O(y^\alpha)$, then
\eqref{calculus2c} remains valid, and \eqref{calculus2a} and
\eqref{calculus2b}
are respectively replaced by $\mathcal
O(t^{(1+\alpha)/2})$ and $\mathcal O(\log t)$.
\end{lem}

We now continue with the proof of Theorem \ref{thm1}. We distinguish two cases.

\subsubsection*{Case 1: $h_0(y)=\mathcal O \Big(y^\nu e^{-\frac v2y}\Big)$ for some $\nu$}
We introduce $H_1(t)$ such that $x H_1(t)$ is the same as $H(x,t)$ with the $\sinh$ expanded to first order:
\begin{equation}\label{defH1}
H_1(t)=
\int_0^\infty\diffd y\,\Big(h_0(y)e^{\frac v2y}\Big)y
e^{\frac{\delta(t)}{2t}y}\psi_t(y)e^{-\frac{y^2}{4t}}.
\end{equation}
For any $z\ge0$, by Taylor's theorem (with the Lagrange remainder), there
exists $w\in[0,z]$ such that $0\le \sinh(z)-z =
\frac{z^3}6\cosh(w)\le\frac{z^3}6e^z$.  It follows that
\begin{equation}
\Big| H(x,t)-x H_1(t)\Big|\le\frac{x^3}{24t^2}\int_0^\infty\diffd
y\,\Big(\big|h_0(y)\big|e^{\frac v2y}\Big)y^3
e^{\frac{x+\delta(t)}{2t}y}\psi_t(y)e^{-\frac{y^2}{4t}}.
\end{equation}
By applying Lemma~\ref{lem:calculus}
to $\phi(y)=\big|h_0(y)\big|e^{\frac v2y}y^3$
with $\alpha=\nu+3$ we obtain
\begin{equation}
H(x,t)-xH_1(t) = \begin{cases}
\mathcal O\left(t^{\nu/2}\right)&\text{if $\nu>-4$}
,\\
\mathcal O\left(t^{-2}\log t\right)&\text{if $\nu=-4$}
,\\
\mathcal O\left({t^{-2}}\right)&\text{if $\nu<-4$}
.
\end{cases}
\label{diffHH1}\end{equation}
We now apply Lemma~\ref{lem:calculus} to $H_1(t)$
with $\alpha=\nu+1$ and obtain 
\begin{equation}
x H_1(t)\sim\begin{cases}
\displaystyle
x\frac A2e^{\Delta}\log t
&\text{if  $h_0(y)\sim A y^{-2} e^{-\frac v2y}$ with $A>0$}
,\\[2ex] \displaystyle
xAe^{\Delta}2^{\nu+1}\Gamma\Big(1+\frac\nu2\Big)t^{1+\frac\nu2}
&\text{if  $h_0(y)\sim A y^\nu e^{-\frac v2y}$ with $A>0$ and
$\nu>-2$}
,\\[2ex] \displaystyle
x \int_0^\infty\diffd y \, h_0(y)y e^{\frac v2y}
\psi_\infty(y)
&\text{if  $h_0(y)=\mathcal O \Big(y^\nu e^{-\frac v2y}\Big)$ for some
$\nu<-2$}
,
\end{cases}
\label{equivH1}
\end{equation}
where we assumed that in the third case the right hand side is non-zero.
As the difference~\eqref{diffHH1} between $H(x,t)$ and $xH_1(t)$
is always asymptotically small compared to the values in the
right hand side of \eqref{equivH1}, it follows that \eqref{equivH1} also gives the asymptotic behaviour of $H(x,t)$. 

We now plug this estimate of $H(x,t)$ into~\eqref{hH}. To prevent
$h\big(m(t)+x,t\big)$ from growing exponentially fast we need to take
$v=2$. Then $\delta(t)$ must be adjusted (up to a constant $a$) to
kill the remaining time dependence.
We find
\begin{equation}
\delta(t)=\begin{cases}
\displaystyle
-\frac{1-\nu}2\log
t+a
+o(1)
&\text{ if  $h_0(y)\sim A y^\nu e^{-y}$ with $A>0$ and $\nu>-2$,}\\[3ex]
\displaystyle
-\frac32\log t+\log\log t +a+o(1)
&\text{ if  $h_0(y)\sim A y^{-2} e^{-y}$ with
$A>0$,}\\[3ex]
\displaystyle
-\frac32\log t+a+o(1)
&\text{ if  $h_0(y)=\mathcal O \big(y^\nu e^{-y}\big)$
for some $\nu<-2$.}
\end{cases}
\end{equation}
In \eqref{hH}, when $t\to\infty$, all the $t$-dependence disappears and
what remains is $\omega(x)$ from the Theorem, with the claimed value of
$\alpha$. This proves cases (b), (c) and (d) of Theorem~\ref{thm1}.

\subsubsection*{Case 2: $h_0(y)\sim A y^\nu e^{-\gamma y}$ with $\gamma<v/2$}

We write $h_0(y)=g_0(y)e^{-\gamma y}$ with $g_0(y)\sim A y^\nu$
so that \eqref{defH} becomes
\begin{equation}
H(x,t)=2t\int_0^\infty\diffd y\,g_0(y)\sinh\Big(\frac{xy}{2t}\Big)
\psi_t(y)e^{\frac{\delta(t)}{2t}y}e^{\frac v 2 y-\gamma y-\frac{y^2}{4t}}.
\end{equation}
The terms in the second exponential reach a maximum at $y=\lambda t$ with
$\lambda=v-2\gamma$. We make the change of variable $y=\lambda t +u\sqrt
t $; after rearranging we have
\begin{equation}
H(x,t)=2t^{\nu+\frac32}e^{\frac{\lambda^2}4t+\lambda\frac{\delta(t)}2}\int_{-\lambda\sqrt
t}^\infty\diffd u\,\frac{g_0(\lambda t+u\sqrt t)}{t^\nu}
\sinh\Big(\frac{\lambda x}{2}+\frac{ux}{2\sqrt t}\Big)
\psi_t(\lambda t+u\sqrt
t)e^{u\frac{\delta(t)}{2\sqrt
t}
-\frac{u^2}{4}}.
\label{tobebound}
\end{equation}
We bound each term in the integral with the goal of applying dominated
convergence.
\begin{itemize}
\item As $g_0$ is bounded for small $y$ and $g_0\sim A y^\nu$ for large $y$, we can
take $\tilde A$ such that $\big|g_0(y)\big|\le \tilde A (y+1)^\nu$. Then
\begin{equation}
\frac{\big|g_0(\lambda t + u\sqrt t )\big|}{t^\nu}\le
\tilde A \lambda^\nu \bigg(1  +\frac{u\sqrt t+1}{\lambda t}\bigg)^\nu
\le\tilde A  \lambda ^\nu e^{\frac{|\nu|(u\sqrt
t+1)}{\lambda t}}\le 2\tilde A\lambda^\nu e^u\ \text{for $t$ large
enough}.
\end{equation}

\item
We have the simple bound
\begin{equation}
\sinh\Big(\frac{\lambda x}{2}+\frac{ux}{2\sqrt t}\Big)
\le e^{\frac{\lambda x}{2}+\frac{ux}{2\sqrt t}}
\le e^{\frac{\lambda x}{2}+u}\ \text{for $t$ large
enough}.
\end{equation}

\item $\psi_t(\cdot)$ is bounded by Proposition~\ref{easybounds}.
\item Finally, $\exp\big(u\delta(t)/(2\sqrt t)\big)\le e^u$ for $t$ large enough.
\end{itemize}
We have bounded the integrand in~\eqref{tobebound}
by a constant times $\exp(3u-u^2/4)$ for $t$
large enough, so we can apply dominated convergence. As $t\to\infty$,
the
$g_0(\cdot)/t^\nu$ term converges to $A\lambda^\nu$, the $\sinh(\cdot)$
term to $\sinh(\lambda x/2)$, the $\psi_t(\cdot)$ term to $e^{\Delta}$
and the exponential to~$e^{-u^2/4}$. We are left with some constants and
the integral of $e^{-u^2/4}$, which is $\sqrt{4\pi}$, and finally:
\begin{equation}
H(x,t)
\sim
2t^{\nu+\frac32}e^{\frac{\lambda^2}4t+\lambda\frac{\delta(t)}2}
A\lambda^\nu
\sinh\Big(\frac{\lambda x}{2}\Big)
e^{\Delta}
\sqrt{4\pi}.
\end{equation}
In~\eqref{hH}, this gives
\begin{equation}
h\big(m(t)+x,t\big)= 
2\sinh\Big(\frac{\lambda x}2\Big)e^{-\frac v2x}
\times
e^{t\big(1-\frac{v^2}4+\frac{\lambda^2}4\big) -\frac
{v-\lambda} 2\delta(t)+o(1)}
t^\nu A\lambda^\nu
.
\end{equation}
Recall that $\lambda=v-2\gamma$. To avoid exponential growth, we need
$1-v^2/4+\lambda^2/4=0$, which implies $v=\gamma+1/\gamma$ with $\gamma<1$
because we started with the assumption $\gamma<v/2$.
As $\frac{v-\lambda}2=\gamma$, to have convergence
of $h\big(m(t)+x,t\big)$
we need $\delta(t)$ to be of the form
\begin{equation}
\delta(t)=\frac \nu\gamma\log t+a+o(1)\quad\text{for large $t$}.
\end{equation}
Writing the $\sinh(\cdot)$ as the difference of two exponentials leads to 
$2\sinh(\lambda x/2)e^{-vx/2}=e^{-\gamma x}-e^{-(1/\gamma)x}$; we then recover case~(a) of Theorem~\ref{thm1} with the claimed value of
$\omega(x)$ and $\alpha$.

This completes the proof of Theorem \ref{thm1}, subject to proving Lemma
\ref{lem:calculus}.\qed

\begin{proof}[Proof of Lemma \ref{lem:calculus}]
Recall from Proposition~\ref{easybounds} that $\psi_t(y)$ is
bounded in $t$ and $y$, 
$\psi_\infty(y):=\lim_{t\to\infty}\psi_t(y)$ exists, 
$\lim_{y\to\infty}\psi_\infty(y)$ exists and equals $e^{\Delta}$, and $\lim_{t\to\infty}\psi_t(t^\alpha)= e^{\Delta}$ for any
$\alpha>0$.

For $\alpha<-1$, the result is obtained with dominated convergence by
noticing that $e^{-y^2/(4t)+\epsilon_t y}$ is bounded by  $e^{t
\epsilon^2_t}$ (value obtained at $y=2t\epsilon_t$). With
$\epsilon_t=o\big(t^{-1/2}\big)$, this is bounded by a constant.

For $\alpha>-1$, cut the integral at $y=1$. The integral from 0 to 1 is
bounded, and in the integral from 1 to $\infty$ we make the substitution $y=u\sqrt t$:
\begin{equation}
\int_0^\infty\diffd y\, \phi(y) e^{-\frac{y^2}{4t}+\epsilon_t y}
\psi_t(y)
=\mathcal O(1)+ t^{\frac{1+\alpha}2}\int_{\frac1{\sqrt t}}^\infty\diffd
u \frac{\phi(u\sqrt t)}{t^{\alpha/2}}e^{-\frac {u^2}{4}+\sqrt t\epsilon_t u}
\psi_t(u\sqrt t).
\end{equation}
A simple application of dominated convergence then leads to
\begin{equation}
\int_0^\infty\diffd y\, \phi(y) e^{-\frac{y^2}{4t}+\epsilon_t y}
\psi_t(y)
=\mathcal O(1)+ t^{\frac{1+\alpha}2}\bigg(\int_0^\infty\diffd
u \, A u^\alpha e^{-\frac {u^2}{4}}e^{\Delta} + o(1)\bigg),
\end{equation}
and the substitution $t=u^2/4$ gives~\eqref{calculus2a}.

For $\alpha=-1$, we cut the integral at $y=\sqrt t$ and again make the
change of variable $y=u\sqrt t$ in the second part:
\begin{equation}
\int_0^\infty\diffd y\, \phi(y) e^{-\frac{y^2}{4t}+\epsilon_t y}
\psi_t(y)
=
\int_0^{\sqrt t}\diffd y\, \phi(y) e^{-\frac{y^2}{4t}+\epsilon_t y}
\psi_t(y)
+
\int_1^\infty\diffd u\,\sqrt t \phi(u\sqrt t) e^{-\frac{u^2}{4}
+\sqrt t\epsilon_t u }
\psi_t(u\sqrt t).
\end{equation}
Again by dominated convergence, the second integral has a limit; we
simply write it as $\mathcal O(1)$. For the first, the integrand is
bounded so the integral from $0$ to $1$ is certainly $\mathcal O(1)$, and
we may concentrate on the integral from $1$ to $\sqrt t$. Making the
substitution $y=t^x$, we have
\begin{equation}
\int_1^{\sqrt t} \diffd y\, \phi(y)
e^{-\frac{y^2}{4t} + \epsilon_t y} \psi_t(y) 
= (\log t)\int_0^{1/2} \diffd x\, t^x \phi(t^x)
e^{-\frac{t^{2x-1}}4+\epsilon_t t^x}
\psi_t(t^x).
\end{equation}
The integrand on the right converges for each $x\in(0,1/2)$ to $A e^\Delta$
so by dominated convergence,
\begin{equation}
\int_1^t  \diffd y \,\phi(y) e^{-\frac{y^2}{4t} + \epsilon_t
y} \psi_t(y)\sim \frac A2  e^{\Delta} \log t,
\end{equation}
as required.
\end{proof}

\section{Estimating $\psi_t$: finer bounds, and Proof of
Theorem~\ref{thm2}}

We want to refine Proposition~\ref{easybounds} and estimate the speed of
convergence of $\psi_t(y,x)$ to its limit as $t\to\infty$. As we are
only interested up to errors of order $\frac{\log t}{t}$, it suffices to
consider the case $x=0$ since by \eqref{noxdep}, $\psi_t(y,x) =
\psi_t(y)e^{x\mathcal O(\frac{\log t}{t})}$.

Recall that
\begin{equation}
       \psi_t(y)     =\E\Big[e^{I_t(y)}\Big],
\qquad \psi_\infty(y)=\E\Big[e^{I(y)}\Big],
\label{psiagain}
\end{equation}
where, introducing $I_t(y):=I_t(y,0)$,
\begin{equation}
\begin{aligned}
I_t(y)&=\frac12\int_0^t\diffd s\, m''(s)\Big(\cbes t y 0_s-y\frac{t-s}t\Big)
=      \frac12\int_0^t\diffd s\, m''(s)\frac{t-s}t\Big(\bes
y_{\frac{st}{t-s}}-y\Big),
\\
I(y)&=\frac12\int_0^\infty\diffd
s\,m''(s)\big(\bes y_s-y\big).
\end{aligned}
\label{Iagain}
\end{equation}
We have used the change of time~\eqref{timechange} to give the second
expression of $I_t(y)$.
As in the hypothesis \eqref{thm2:condm1} of Theorem~\ref{thm2}, we suppose that $m$ is twice
continuously differentiable and 
\begin{equation}
m''(t) = \frac 3{2(t+1)^2} +r''(t) \quad \text{ with } \quad r''(t)=\mathcal O\left(\frac
1{t^{2+\eta}}\right), \quad \eta>0.
\label{newmprop}
\end{equation}

Our estimate of $\psi_t(y)$ is based on the following two propositions. By
writing $I_t(y) = I(y) - (I(y)-I_t(y))$ in the definition of $\psi_t(y)$,
and expanding the exponential in the small correction term $I(y)-I_t(y)$,
we show that:

\begin{prop}
\label{finer}
Assuming \eqref{newmprop}, the following holds uniformly in $y$:
\begin{equation}
\psi_t(y)=\psi_\infty(y)\Big(1-\E\big[I(y)-I_t(y)\big]\Big)+\mathcal
O\Big(\frac{\log t}{t}\Big) + y\mathcal O\Big(\frac{1}{t}\Big).
\label{eq:finer}
\end{equation}
\end{prop}

Further, some straightforward computations give that:

\begin{prop}
\label{finer2}
Assuming \eqref{newmprop}, the following holds uniformly in $y$:
\begin{equation}
\E\big[I(y)-I_t(y)\big] = \frac{3\sqrt\pi}{\sqrt t} 
+y\mathcal O\left(\frac{\log t}t\right)
+\begin{cases} 
\displaystyle
\mathcal O\left(\frac{1}{t} \right)& \text{if $\eta>1/2$}
,\\[1.5ex] \displaystyle
\mathcal O\left(\frac{\log t}{t} \right)& \text{if $\eta=1/2$}
,\\[1.5ex] \displaystyle
\mathcal O\left(\frac1{t^{1/2+\eta}}\right) & \text{if $\eta<1/2$}.
\end{cases}
\end{equation}
\end{prop}

We prove Propositions~\ref{finer} and~\ref{finer2} in Sections~\ref{prooffiner}
and~\ref{prooffiner2}, after some preparatory work in Section \ref{decorsec}. We now show how to prove Theorem~\ref{thm2} from these two propositions.

\begin{proof}[Proof of Theorem~\ref{thm2}]
We assume that $m(t)$ satisfies the hypothesis~\eqref{thm2:condm1} of
Theorem~\ref{thm2}:
\begin{equation}
m(t)=2t-\frac32\log (t+1) +a + r(t)\quad\text{with $r(t)=o(1)$ and
$r''(t)=\mathcal O\Big(\frac1{t^{2+\nu}}\Big)$ for large~$t$}.
\label{m(t)finer}
\end{equation}
As in the proof of Theorem~\ref{thm1}, we recall that
$h\big(m(t)+x,t\big)$ is related to $H(x,t)$ through~\eqref{hH} and that
$H(x,t)$ is given by \eqref{defH}. With $v=2$ and $\delta(t)=-(3/2)\log
(t+1)+a+r(t)$, these two equations read:
\begin{gather}
h\big(m(t)+x,t\big)=\frac1{\sqrt{4\pi}}e^{-a-r(t)-\Delta-x+\mathcal
O\big(\frac{\log t}t\big)}H(x,t),\label{hHagain}\\
H(x,t)=\int_0^\infty\diffd
y\,\Big(h_0(y)e^y\Big)2t\sinh\Big(\frac{xy}{2t}\Big
)e^{\frac{-(3/2)\log (t+1) +a+r(t)}{2t}y-\frac{y^2}{4t}}\psi_t(y),
\end{gather}
We compute $H(x,t)$ for an initial
condition $h_0(x)=\mathcal O\big(x^\nu e^{-x}\big)$ for some $\nu<-2$.
In~\eqref{defH1} in the proof of Theorem~\ref{thm1}, we introduced 
$H_1(t)$ which is $H(x,t)/x$ with the sinh replaced by its first order
expansion:
\begin{equation}
H_1(t)=
\int_0^\infty\diffd y\,\Big(h_0(y)e^{y}\Big)y
e^{\frac{-(3/2)\log (t+1) +a+ r(t)}{2t}y-\frac{y^2}{4t}}\psi_t(y),
\end{equation}
and we showed in~\eqref{diffHH1} that the difference between $H(x,t)$ and
$xH_1(t)$ is very small. We continue to simplify the integral by introducing successive simplifications
\begin{equation}\begin{aligned}
H_2(t)&=
\int_0^\infty\diffd y\,\Big(h_0(y)e^{y}\Big)y
e^{-\frac{y^2}{4t}}\psi_t(y),\\H_3(t)&=
\int_0^\infty\diffd y\,\Big(h_0(y)e^{y}\Big)y
e^{-\frac{y^2}{4t}}\psi_\infty(y),\\
H_4&= \int_0^\infty\diffd y\,\Big(h_0(y)e^{y}\Big)y \psi_\infty(y),
\end{aligned}\end{equation}
and by writing
\begin{multline}
H(x,t)=\Big(H(x,t)-xH_1(t)\Big)+x\Big(H_1(t)-H_2(t)\Big)+x\left(H_2(t)-\left[1-\frac{3\sqrt\pi}{\sqrt t}\right]H_3(t)\right)
\\+x\left[1-\frac{3\sqrt\pi}{\sqrt t}\right]\Big(H_3(t)-H_4\Big)+x\left[1-\frac{3\sqrt\pi}{\sqrt t}\right]H_4.\label{coll}
\end{multline}
We now bound the successive differences in the above expression, as we did in \eqref{diffHH1}, for the first one.

For $t$ large enough, $-\frac32\log (t+1)+a +r(t)<0$ and for $z>0$ we have $0\le 1-e^{-z}\le z$. Thus
\begin{equation}
\Big| H_2(t)-H_1(t)\Big|
\le \frac{\frac32\log (t+1) -a -r(t)}{2t}
\int_0^\infty\diffd y\,\Big(\big|h_0(y)\big|e^{y}\Big)y^2
e^{-\frac {y^2}{4t}}\psi_t(y).
\end{equation}
An application of Lemma~\ref{lem:calculus} with $\phi(y)=h_0(y)e^{y}y^2$
and hence $\alpha=\nu+2$ then gives
\begin{equation}\label{diffH1H2}
H_1(t)-H_2(t)=\begin{cases}
\mathcal O\left(t^{\frac{1+\nu}2}\log t\right)&\text{if $\nu>-3$}
,\\[1ex]
\mathcal O\left(\frac{\log^2 t}t\right)&\text{if $\nu=-3$}
,\\[1ex]
\mathcal O\left(\frac{\log t}{t}\right)&\text{if $\nu<-3$}
.
\end{cases}
\end{equation}

For the difference involving $H_2$ and $H_3$, we use
Propositions~\ref{finer} and~\ref{finer2} which give that
uniformly in
$y$,
\begin{equation}
\label{eq:finer2}
\psi_t(y)=\psi_\infty(y)\left(1-\frac{3\sqrt\pi}{\sqrt t}\right)
+y \mathcal O\left(\frac{\log t}{t}\right)
+\begin{cases} 
\mathcal O\left(\frac{1}{t} \right)& \text{if $\eta>1/2$}
,\\[1.5ex]
\mathcal O\left(\frac{\log t}{t} \right)& \text{if $\eta=1/2$}
,\\[1.5ex]
\mathcal O\left(\frac1{t^{1/2+\eta}}\right) & \text{if $\eta<1/2$}
.
\end{cases}
\end{equation}
We get
\begin{equation}\label{diffH2H3}
\begin{aligned}
H_2(t)-\left(1-\frac{3\sqrt\pi}{\sqrt t}\right)H_3(t)
&=\int_0^\infty\diffd y\,\Big(h_0(y) e^y\Big)
y e^{-\frac{y^2}{4t}}\left[
\psi_t(y)-\psi_\infty(y)\left(1-\frac{3\sqrt\pi}{\sqrt t}\right)
\right],\\
&=\mathcal
O\left(\frac1{t^{1/2+\eta}}\right)
+\begin{cases}
\mathcal O\left(t^{\frac{1+\nu}2}\log t\right)&\text{if $\nu>-3$}
,\\[1ex]
\mathcal O\left(\frac{\log^2 t}t\right)&\text{if $\nu=-3$}
,\\[1ex]
\mathcal O\left(\frac{\log t}{t}\right)&\text{if $\nu<-3$}
.
\end{cases}
\end{aligned}
\end{equation}
Indeed, the $y\mathcal O\big(\frac{\log t}{t}\big)$ gives the same correction as
in~\eqref{diffH1H2} by another application of Lemma~\ref{lem:calculus} with
$\alpha=\nu+2$. As $\int\diffd y\,\big|h_0(y)\big|e^y y<\infty$ because
$\nu<-2$, the
contribution of the $y\mathcal O\big(\frac{\log t}{t}\big)$ term subsumes the
other $\mathcal O$ in \eqref{eq:finer2} except in the case
$\eta<\frac12$.

Finally, notice that $\big|H_4\big|<\infty$ because we supposed $\nu<-2$.
Recalling $\psi_\infty(y)\le K_2$, one has
\begin{equation}\label{diffH3H4}
\begin{aligned}
\Big| H_4-H_3(t)\Big|&\le
\int_0^\infty\diffd
y\, \Big(\big|h_0(y)\big|e^y\Big)y
	\Big(1-e^{-\frac{y^2}{4t}}\Big)\psi_\infty(y),
\\ &\le K_2
\int_0^{\sqrt t}
\diffd y\, \Big(\big|h_0(y)\big|e^y\Big)y\frac{y^2}{4t}
+K_2\int_{\sqrt t}^{\infty}
\diffd y\, \Big(\big|h_0(y)\big|e^y\Big)y
,
\\ &=
\begin{cases}
\mathcal O\left(t^{1+\frac\nu2}\right)&\text{if $-2>\nu>-4$}
,\\[1ex]
\mathcal O\left(\frac{\log t}t\right)&\text{if $\nu=-4$}
,\\[1ex]
\mathcal O\left(\frac1t\right)&\text{if $\nu<-4$}
,
\end{cases}
\end{aligned}
\end{equation}
where we used $h_0(y)e^y=\mathcal O(y^\nu)$. The end result comes from the
integral from 0 to $\sqrt t$; the other integral is always $\mathcal
O(t^{1+\nu/2})$.

Finally, collecting the differences \eqref{diffHH1}, \eqref{diffH1H2},
\eqref{diffH2H3} and \eqref{diffH3H4} leads  with \eqref{coll} to
\begin{equation}
\label{collect}
H(x,t)=x H_4\left[1-\frac{3\sqrt\pi}{\sqrt t}
+\mathcal O \left({t^{1+\frac\nu2}}\right)
+\mathcal O \left(\frac1{t^{1/2+\eta}}\right)
+\mathcal O\left(\frac{\log t}{t}\right)\right].
\end{equation}
Substituting into \eqref{hHagain} and expanding $e^{-r(t)}$
leads to the main expression  \eqref{thm2:main} of Theorem~\ref{thm2}, with
the value $\alpha$ given in Theorem~\ref{thm1}.

We now turn to the second part of Theorem~\ref{thm2} and
assume that $h_0(y)\sim A y^\nu e^{-y}$ with $-4<\nu<-2$. We look for an
estimate of $H_4-H_3(t)$ which is more precise than~\eqref{diffH3H4}.

Writing  $H_4-H_3(t)$ as a single integral and doing the change
of variable $y=u\sqrt t$ one gets
\begin{equation}\label{diffH3H4bis}
H_4-H_3(t)= t^{1+\frac\nu2} \int_0^\infty\diffd u\,\frac{h_0(u\sqrt
t)e^{u\sqrt t}}{t^{\nu/2}}u\Big(1-e^{-\frac{u^2}4}\Big)\psi_\infty(u\sqrt
t).
\end{equation}
A simple application of dominated convergence then gives
\begin{equation}
H_4-H_3(t)\sim t^{1+\frac\nu2} Ae^{\Delta}\int_0^\infty\diffd
u\,u^{\nu+1}
\Big(1-e^{-\frac{u^2}4}\Big)
= - Ae^{\Delta}
2^{\nu+1}\Gamma\Big(\frac\nu2+1\Big)t^{1+\frac\nu2},
\end{equation}
and \eqref{collect} becomes
\begin{equation}
H(x,t) = x H_4\left[1-\frac{3\sqrt\pi}{\sqrt t}
\right]
+ xAe^{\Delta}2^{\nu+1}\Gamma\Big(\frac\nu2+1\Big)t^{1+\frac\nu2}
+o\Big(t^{1+\frac\nu2}\Big)
+\mathcal O\left(\frac1{t^{\frac12+\eta}}\right)
.
\end{equation}
This leads with \eqref{hHagain} to \eqref{thm2:main}.
\end{proof}

\subsection{Decorrelation between $I(y)$ and $\bes y_s$}\label{decorsec}

A large part of our argument relies on a statement that roughly says
``$I(y)$ and $\bes y_s$ are almost independent for large $s$''. The
following proposition makes this precise.

\begin{prop}\label{wlem}
Suppose that $m$ is twice continuously differentiable with $m''(t)
=\mathcal O({1}/{t^2})$. 
Define
\begin{equation}\label{wdef}
w(y,s) = \E\Big[e^{I(y)}\big(\bes y _s-y\big)\Big]
        -\E\Big[e^{I(y)}\Big]\E\big[\bes y _s-y\big].
\end{equation}
There exists a constant $C>0$ such that
\begin{equation}
\begin{aligned}
|w(y,s)| &\le C\log(s+1) &&\hbox{ for all $s,y\ge 0$,}\\
|w(y,s)| &\le C(1+y\frac{\log(s+1)}{\sqrt s}) &&\hbox{ for all $s,y\ge 0$,}\\
\Big|w(y,s+\delta)-w(y,s)\Big|&\le C\frac{\delta}{s+1} &&\hbox{for all
$y\ge0$, whenever $0\le\delta\le s^2$}.
\end{aligned}
\label{interdeco}
\end{equation}
\end{prop}

The proof of this result is quite involved. The first step is to prove
two fairly accurate estimates on the difference between two bridges with
different end points, the first of which is best when the starting point
$y$ is large and the second of which is more accurate when $y$ is small.

It is well-known that a Bessel process started from $y$ and conditioned to
be at position~$x$ at time~$t$ is equal in law to a Bessel bridge from $y$
to $x$ in time $t$ followed by an independent Bessel process started from
$x$ at time~$t$. We defined $\cbes tyx_s$ for $s\in[0,t]$ as a Bessel
bridge from $y$ to $x$ in a time~$t$. In this section, we extend the
definition of $\cbes tyx_s$ for $s>t$ by interpreting  it as an independent
Bessel started from $x$ at time $t$, so that $\cbes tyx_s, s\ge0$ is
a Bessel process conditioned to be at $x$ at time~$t$. We assume that the
Bessel processes attached to $\cbes tyx_s$ for $s\ge t$ are built for all
$x$ and $t$ with the same noise, so that we can compare them to each other.
In particular, we apply \eqref{beslem1} and
\eqref{bessfluc:1} to these Bessel processes.

Recall that $I(y)=\frac12\int_0^\infty\diffd u\,m''(u)\big(\bes y_u-y\big)$
and define
\begin{equation}\tilde I_t(y,z) = \frac12 \int_0^\infty\diffd u\,m''(u)\big(\cbes{t}{y}{z}_u-y\big).
\label{defItilde}
\end{equation}

\begin{lem}\label{intdiff}
If $m$ is twice continuously differentiable with $m''(t)=\mathcal O(1/t^2)$, then there exists a constant $c$ and random variables $G_t$
with distribution independent of $t$ and Gaussian
tails such that:
\begin{itemize}
\item
For any  $t$, $y$, $z$ and $x$,
\begin{equation}\label{easyJ}
|\tilde I_t(y,z)-\tilde I_t(y,x)|\le c|z-x|\frac{\log(t+1)}{t}.
\end{equation}
\item For any $t$, $y$ and $z$,
\begin{equation}\label{hardIdif}
\left| \tilde I_t(y,z) - \tilde I_t(y,0) \right|
 \leq \frac{z^2}{t^{3/2}}G_t + c\bigg(\frac{z}{t} + \frac{z^3}{t^2} +  
\frac{z^2 y}{t^2}\log(t+1)\bigg).
\end{equation}
\end{itemize}
\end{lem}

\begin{proof}
Recall from \eqref{beslem1} and \eqref{beslem2} that $\big| \cbes t y z_s - \cbes t y x_s\big|
 \le |z-x|\,\min(s/t,1)$. Therefore
\begin{align}
\left| \tilde I_t(y,z) - \tilde I_t(y,x) \right| 
&\leq\frac12
\int_0^\infty \diffd s\, |m''(s)|\, \left|\cbes{t}{y}{z}_s - \cbes{t}{y}{x}_s\right|
\\ &\leq \frac12|z-x| \bigg(\int_0^t \diffd s\, |m''(s)| \frac{s}t
+ \int_t^\infty\diffd s\, |m''(s)|\bigg) .
\end{align}
The first integral is $\mathcal O\big(\frac{\log t}{t}\big)$ while the
second is a $\mathcal O(1/t)$. Their sum can be bounded by
$2c\log(t+1)/t$ for some $c$, which proves the simpler
bound~\eqref{easyJ}.

To prove \eqref{hardIdif} we consider $x=0$ and split the integral at
$t/2$ and $t$. For $s>t/2$, with the same simple bounds as above we have
\begin{equation}
\bigg|\int_{\frac t2}^\infty \diffd s\, m''(s)\, \left(\cbes{t}{y}{z}_s
- \cbes{t}{y}{0}_s\right)\bigg|
 \leq z\bigg(\int_{\frac t2}^t \diffd s\, |m''(s)| \frac{s}t
+ \int_t^\infty\diffd s\, |m''(s)|\bigg)=z\mathcal O\Big(\frac1t\Big) .
\label{notkeyIdiff}
\end{equation}
From $0$ to $t/2$, we claim that the following bound is true:
\begin{equation}\label{keyIdiff}
0\le \int_0^{\frac t2}\diffd s \frac{\cbes{t}{y}{z}_s - 
\cbes{t}{y}{0}_s}{(1+s)^2} \le \frac{z^3}{3t^2} + \frac{z^2}{t^{3/2}}G_t
+ \frac{2z^2 y}{3t^2}\log(t+1),
\end{equation}
for some non-negative $G_t$ with distribution
independent of $t$ and Gaussian tails. Then, as there exists some constant $c'$ such that
$|m''(s)|\le c'/(1+s)^2$, \eqref{notkeyIdiff} and \eqref{keyIdiff} give the
result~\eqref{hardIdif}. Therefore it only remains to prove
\eqref{keyIdiff}.

We use the bound $\coth(x)\leq 1/x + x/3$, together with the SDEs
\eqref{introcbes0} and \eqref{introcbes}. We already know from Lemma
\ref{beslem} that $\cbes t y 0_s \leq \cbes t y z_s \leq \cbes t y 0_s
+ zs/t$ for any $s\in[0,t]$. Therefore for any $s\in[0,t)$,
\begin{align}
\diffd\cbes t y z_s - \diffd\cbes t y 0_s &\leq \left(\frac{z}{t-s} \coth\frac{z\cbes t y z_s}{2(t-s)} - \frac{2}{\cbes t y 0_s}\right)\diffd s\\
&\leq \frac{z^2 \cbes t y z_s}{6(t-s)^2}\diffd s\\
&\leq \left(\frac{z^3 s}{6t(t-s)^2} + \frac{z^2 \cbes
t y 0_s}{6(t-s)^2}\right)\diffd s.
\end{align}

By integration by parts,
\begin{equation}\int_0^{\frac t2} \diffd s\, \frac{\cbes t y z_s}{(s+1)^2}
= \int_0^{\frac t2}\frac{1}{s+1}\diffd \cbes t y z_s - \frac{\cbes
t y z_{t/2}}{t/2+1} + y
.
\end{equation}
Using \eqref{beslem2}, the estimate on $\diffd\cbes t y z_s
- \diffd\cbes t y 0_s$
from above, and $t-s\ge t/2$ for $s\le t/2$, we get
\begin{align}
0\le\int_0^{\frac t2} \diffd s\, \frac{\cbes t y z_s-\cbes t y 0_s}
{(s+1)^2}
&\leq \int_0^{\frac t2}
\frac{1}{s+1} \diffd \cbes t y z_s - \int_0^{\frac t2} \frac{1}{s+1} \diffd
\cbes t y 0_s
\\
&\leq \int_0^{\frac t2} \diffd s\, \frac{z^3 s}{6t(t-s)^2(s+1)}
+ \int_0^{\frac t2}
\diffd s\, \frac{z^2 \cbes t y 0_s}{6(t-s)^2(s+1)}
\\&\leq \frac{2z^3}{3t^3}\int_0^{\frac t2} \diffd s\, \frac{s}{s+1}
+ \frac{2z^2}{3t^2}\int_0^{\frac t2} \diffd s\, \frac{\cbes t y 0_s}{s+1}
\\&\leq \frac{z^3}{3t^2}+ \frac{2z^2}{3t^2}\int_0^{\frac t2} \diffd s\,
\frac{y+\cbes t 0 0_s}{s+1}
\\&\leq \frac{z^3}{3t^2}+ \frac{2z^2y}{3t^2}\log(t+1)+
 \frac{2z^2y}{3t^2}\int_0^{\frac t2} \diffd s\, \frac{\cbes t 0 0_s}{s}.
\label{whole}
\end{align}
By the scaling property, we introduce another Bessel bridge
$\tcbes 100$ by setting $\cbes t00_{tu}=\sqrt t \tcbes100_u$.
By adapting Lemma~\ref{bessfluc} to the new Bessel bridge,
there exists a random variable $G_t$ with distribution independent of $t$ and Gaussian tails such that $\tcbes100_u\le G_t u^{\frac14}$. Hence
\begin{equation}
\int_0^{\frac t2}\diffd s\frac{\cbes t00_s}s
=\int_0^{\frac12}\diffd u\frac{\cbes t00_{tu}}u
\le \sqrt t\,G_t\int_0^{\frac12}\diffd u\,u^{-\frac34}\le4G_t\sqrt t.
\end{equation}
This bounds the last term in \eqref{whole} and establishes
\eqref{keyIdiff}, thereby completing the proof.
\end{proof}

Finally, given that we are using random variables with Gaussian tails, the
following trivial result is useful.

\begin{lem}\label{gausslem}
Suppose that $G$ is a random variable with Gaussian tails.
Then for any real number $a$ and any polynomial $P$,
\begin{equation}
\big|\E[P(G) e^{aG}]\big|<\infty.
\end{equation}
\end{lem}

We can now prove Proposition \ref{wlem}.

\begin{proof}[Proof of Proposition \ref{wlem}]
Recall the definition\eqref{defItilde} of $\tilde I$.
For any deterministic $x$, since
$\E\big[\bes{y}_s-\E\big[\bes{y}_s\big]\big] = 0$ and $\E\big[e^{\tilde
I_s(y,x)}\big]$ is deterministic, we have
\begin{align}
w(y,s) &=
\E\Big[e^{I(y)}\Big(\bes{y}_s-\E\big[\bes{y}_s\big]\Big)\Big]
\\ &=
\E\Big[\Big(e^{I(y)}-\E\big[e^{\tilde I_s(y,x)}\big]\Big)
	\Big(\bes{y}_s-\E\big[\bes{y}_s\big]\Big)\Big]
\\ &=
\int_0^\infty \Big(\E\big[e^{ I(y)}|\bes{y}_s=z\big]-
\E\big[e^{\tilde I_s(y,x)}\big]\Big)\Big(z-\E\big[\bes{y}_s\big]\Big)
\P\big(\bes{y}_s \in \diffd z\big)
\\&=
\int_0^\infty \E\Big[e^{ \tilde I_s(y,z)}-
e^{\tilde I_s(y,x)}\Big]\Big(z-\E\big[\bes{y}_s\big]\Big)
\P\big(\bes{y}_s \in \diffd z\big),
\end{align}
where we used  that $\E\big[e^{I(y)}\big|\bes y_s = z\big] = \E\big[ e^{\tilde
I_s(y,z)}\big]$. Then
\begin{equation}
\big|w(y,s)\big|
\le
\int_0^\infty \E\Big[\,\Big|e^{ \tilde I_s(y,z)}-
e^{\tilde I_s(y,x)}\Big|\,\Big]\times\Big|z-\E\big[\bes{y}_s\big]\Big|
\times
\P\big(\bes{y}_s \in \diffd z\big).
\end{equation}
By the mean value theorem, $|e^a-e^b|\le |a-b|e^{\max(a,b)}\le
|a-b|e^{b+|a-b|}$. Thus
\begin{align}
|w(y,s)| &\le
\int_0^\infty \E\bigg[e^{\tilde I_s(y,x)}\,\Big|\tilde
I_s(y,z)-\tilde I_s(y,x)\Big|\,e^{|\tilde I_s(y,z)-\tilde
I_s(y,x)|}\bigg]\;\Big|z - \E[\bes{y}_s]\Big|\times\P\big(\bes{y}_s \in \diffd
z\big)
\\&\le
\int_0^\infty \E\Big[e^{\tilde I_s(y,x)}\,\Big|\tilde
I_s(y,z)-\tilde I_s(y,x)\Big|\Big]\,e^{c|z-x|\frac{\log(s+1)}s}
\;\Big|z - \E[\bes{y}_s]\Big|\times\P\big(\bes{y}_s \in \diffd
z\big),
\end{align}
where we applied \eqref{easyJ} of Lemma~\ref{intdiff} in the exponential. Now,
by Cauchy-Schwarz,
\begin{equation}\label{kille1}
|w(y,s)| \le \E\Big[e^{2\tilde I_s(y,x)}\Big]^{\frac12}
\int_0^\infty \E\Big[\Big|\tilde I_s(y,z)-\tilde I_s(y,x)\Big|{}^2
\Big]^{\frac12}e^{c|z-x|\frac{\log(s+1)}s}\,\Big|z - \E[\bes{y}_s]\Big|\P(\bes{y}_s \in \diffd z).
\end{equation}
Decompose $\tilde I_s(y,x)$ in the following way:
\begin{multline}
2\tilde I_s(y,x)=\int_0^s\diffd u\,m''(u)\Big(\cbes
syx_u-y-(x-y)\frac us\Big)+\int_s^\infty\diffd u\,m''(u)\Big(\cbes
syx_u-x\Big)\\+\int_0^s\diffd u\,m''(u)(x-y)\frac us+\int_s^\infty\diffd
u\,m''(u)(x-y).
\end{multline}
The first integral is $2I_s(y,x)$. Using~\eqref{bessfluc:2} it can be
bounded uniformly in $y$, $x$ and $s$ by a variable with Gaussian tails.
The second integral, which does not depend on $y$,
can also be bounded uniformly in $x$ and $s$ using \eqref{bessfluc:1} by an
independent variable with Gaussian tails. The third integral is
$(x-y)\mathcal O\big(\frac{\log s}s\big)$ and the fourth
is $(x-y)\mathcal O\big(\frac{1}s\big)$; they can be bounded together
by $2c|x-y|\frac{\log (s+1)}s$ for some constant $c$. 
Finally, there exists a $C_1$ and a $c$ such that, uniformly in $s$, $y$
and $x$: 
\begin{equation} \E\big[e^{2\tilde I_s(y,x)}\big]^{\frac12}\le 
C_1 e^{c|x-y|\frac{\log (1+s)}s}.
\label{eJbdd}
\end{equation}
Substituting back into \eqref{kille1}, we get
\begin{align}
|w(y,s)|&\le C_1 e^{c|x-y|\frac{\log (1+s)}s}
\int_0^\infty \E\Big[\Big|\tilde I_s(y,z)-\tilde I_s(y,x)\Big|{}^2
\Big]^{\frac12}e^{c|z-x|\frac{\log (1+s)}s}\,\Big|z - \E[\bes{y}_s]\Big|\P(\bes{y}_s \in \diffd z).
\label{kille}
\end{align}

First we concentrate on showing the first line of
\eqref{interdeco}, i.e.~that $|w(y,s)|\le C\log(s+1)$. Using
\eqref{easyJ} again,
\begin{equation}\E[|\tilde I_s(y,z)-\tilde I_s(y,x)|^2]^{1/2} \le
c|z-x|\frac{\log(s+1)}{s},
\end{equation}
so we get, by choosing $x=\E\big[\bes{y}_s\big]$,
\begin{align}
|w(y,s)| &\le
C_1c\frac{\log(s+1)}s e^{c\big|\E[\bes y_s]-y\big|\frac{\log (1+s)}s}
\int_0^\infty 
e^{c\big|z-\E[\bes y_s]\big|\frac{\log (1+s)}s}\,\Big(z
- \E[\bes{y}_s]\Big)^2\P(\bes{y}_s \in \diffd z),\notag
\\&=
C_1c\frac{\log(s+1)}s e^{c\big|\E[\bes y_s]-y\big|\frac{\log (1+s)}s}
\E\Big[
e^{c\big|\bes y_s-\E[\bes y_s]\big|\frac{\log (1+s)}s}\,
\Big(\bes y_s -\E[\bes{y}_s]\Big){}^2\Big].
\end{align}

It remains to bound the expectations above.
Note from \eqref{beslem1} that for all $z\ge0$
we have $B_1 \le \bes z_1 - z \le \bes 0_1$ and therefore
\begin{equation}
B_1-\E[\bes 0_1] \le \bes z_1 -\E[\bes z_1]
                 \le \bes z_1 -z \le \bes 0_1,
\end{equation}
so, with $\Gamma$ the positive random variable with Gaussian tail
defined by
\begin{equation}
\Gamma:=\max\big\{\big|B_1-\E\big[\bes 0_1\big]\big|,\big|\bes
0_1\big|\big\},
\label{gamdef}
\end{equation}
we have, uniformly in $z$,
\begin{equation}
\big|\bes z_1 - \E\big[\bes z_1\big]\big|\le \Gamma,
\qquad
\big|\bes z_1 - z\big|\le \Gamma.
\label{gamuse}
\end{equation}
Therefore, by the scaling property,
\begin{equation}
|w(y,s)| \le
C_1c\frac{\log(s+1)}s e^{c\sqrt s\,\E[\Gamma]\frac{\log (1+s)}s}
\E\Big[ e^{c\sqrt s\,\Gamma \frac{\log (1+s)}s} s \Gamma^2\Big]
\le C \log(s+1),
\end{equation}
for some constant $C$, where we used Lemma~\ref{gausslem} to bound the
last expectation.
This is the first line of \eqref{interdeco}.

We now turn to showing the second line of \eqref{interdeco}, that
$|w(y,s)|\le C\big(1+y\frac{\log(s+1)}{\sqrt s}\big)$. Given that we have already
proven that $|w(y,s)|\le C\log(s+1)$, it suffices to consider $y\le
\sqrt s$.

Recall \eqref{hardIdif}:
\begin{equation}
\left| \tilde I_s(y,z) - \tilde I_s(y,0) \right|
 \leq \frac{z^2}{s^{3/2}}G_s + c\bigg(\frac{z}{s} + \frac{z^3}{s^2} +  
\frac{z^2 y}{s^2}\log(s+1)\bigg).
\end{equation}
By Cauchy-Schwarz, if $a,b\ge 0$ and $X$ is a non-negative random
variable with finite second moment, then
\begin{equation}\E[(aX+b)^2]^{1/2} \le a\E[X^2]^{1/2}+b.\end{equation}
This tells us that
\begin{equation}\label{killnoe}
\E[|\tilde I_s(y,z)-\tilde I_s(y,0)|^2]^{\frac12} \le C_2 a_{s,y,z}
\qquad\text{with }a_{s,y,z}=\frac{z}{s} +\frac{z^2}{s^{3/2}} +
\frac{z^3}{s^2} +  \frac{z^2 y}{s^2}\log(s+1)
\end{equation}
for some constant $C$ since the distribution of $G_s$ does not depend on $s$. 

Now choosing $x=0$ in \eqref{kille} and substituting \eqref{killnoe}, we get
\begin{align}
|w(y,s)|&\le C_1C_2 e^{cy\frac{\log (1+s)}s}
\int_0^\infty a_{s,y,z}e^{cz\frac{\log (1+s)}s}\,\Big|z - \E[\bes{y}_s]\Big|\P(\bes{y}_s \in \diffd z).
\\&\le C_3 \E\Big[a_{s,y,\bes y_s}\,e^{c\bes
y_s\frac{\log (1+s)}s}
	\Big|\bes y_s - \E[\bes{y}_s]\Big|\Big],
\end{align}
where we used $y\le\sqrt s$ to bound the factor in front of the integral by
a constant.
Using the scaling property, writing $\tilde\xi_1=\bes y_s/\sqrt s$ we have
as in \eqref{gamuse}
\begin{equation}
\big|\tilde\xi_1-\E[\tilde\xi_1]\big|\le\Gamma,\qquad
\big|\tilde\xi_1-y/\sqrt s\big|\le\Gamma,\qquad
\tilde\xi_1\le 1+\Gamma
\end{equation}
for some positive random variable $\Gamma$ with Gaussian tails;
we used $y\le\sqrt s$ in the last equation.
Then
\begin{align}
|w(y,s)|&
\le C_3 \E\Big[a_{s,y,\sqrt s(1+\Gamma)}
e^{c\sqrt s(1+\Gamma)\frac{\log (1+s)}s}
	\sqrt s\, \Gamma\Big],
\end{align}
but
\begin{equation}
a_{s,y,\sqrt s \,X}\sqrt s=X+X^2+X^3+ X^2y\frac {\log(s+1)}{\sqrt s},
\end{equation}
so using Lemma~\ref{gausslem} again we obtain $|w(y,s)|\le C\Big(1+y\frac {\log(s+1)}{\sqrt s}\Big)$ 
for some constant $C$, which is the second line
of \eqref{interdeco}.

Finally we turn to the last line of~\eqref{interdeco} and bound the
increments of  $w(y,s)$. Our approach is very similar to the above,
conditioning on the value of $\bes y_{s+\delta}-\bes y_s$ instead of $\bes
y_s$.

Let $X = \bes y_{s+\delta}-\bes y_s$ and $\mu = \E[X]$, and also define
\begin{equation}
\mathcal E(x) = \E\Big[e^{I(y)}\Big|X = x\Big].
\end{equation}
Directly from the definition \eqref{wdef} of $w$, since $\mathcal E(\mu)$ is
deterministic and $\E[X-\mu]=0$, we have
\begin{align}
w(y,s+\delta)-w(y,s) &= \E\Big[e^{I(y)}(X - \mu)\Big],\\
&= \E\Big[\Big(e^{I(y)}-\mathcal E(\mu)\Big)(X - \mu)\Big],\\
&= \int_{-\infty}^\infty \big(\mathcal E(x) - \mathcal E(\mu)\big) (x-\mu)\P(X\in\diffd x).\label{wtoE}
\end{align}
Applying the Markov property at time $s$, we have
\begin{align}\label{twoEs}
\mathcal E(x)-\mathcal E(\mu)
= \int_0^\infty& \P(\bes y_s\in \diffd z)
\E\Big[e^{\frac12\int_0^s \diffd u\, m''(u)\big(\cbes syz_u - y\big)}\Big]\notag \\
&\cdot\E\Big[e^{\frac12\int_0^\infty \diffd u\, m''(s+u)\big(\cbes \delta
z {z+x}_u - y\big)} - e^{\frac12\int_0^\infty \diffd u\, m''(s+u)\big(\cbes \delta
z {z+\mu}_u - y\big)}\Big].
\end{align}

We now use the simple bound
\begin{equation}
|\cbes{\delta}{z}{z+x}_u - \cbes{\delta}{z}{z+x'}_u|\le |x-x'| \quad \text{
for all } z,x,x',\delta,u\ge0,
\end{equation}
which follows from Lemma \ref{beslem} and implies that
\begin{equation}
\Big|\int_0^\infty \diffd u\, m''(s+u) (\cbes \delta z {z+x}_u - \cbes
\delta z {z+\mu}_u)\Big| \le \int_0^\infty\diffd u\,
|m''(s+u)||x-\mu| \le \frac{2c|x-\mu|}{s+1}
\end{equation}
for some constant~$c$.
This, together with the bound $|e^a-e^b|\le |a-b|e^{b+|a-b|}$ for any $a,b\in\R$, tells us that
\begin{multline}
\Big|e^{\frac12\int_0^\infty \diffd u\, m''(s+u)\big(\cbes \delta z {z+x}_u
- y\big)} - e^{\frac12\int_0^\infty \diffd u\, m''(s+u)\big(\cbes \delta
z {z+\mu}_u - y\big)}\Big|\\
\le e^{\frac12\int_0^\infty \diffd u\, m''(s+u)\big(\cbes \delta z {z+\mu}_u
- y\big)} \frac{c|x-\mu|}{s+1} e^{\frac{c|x-\mu|}{s+1}}.
\end{multline}
Substituting this into \eqref{twoEs}, we have
\begin{align}
|\mathcal E(x)-\mathcal E(\mu)| 
&\le \int_0^\infty \P(\bes y_s\in \diffd z)
\E\Big[e^{\frac12\int_0^s \diffd u\, m''(u)\big(\cbes syz_u - y\big)}\Big]\notag\\
&\hspace{15mm}\cdot\E\Big[e^{\frac12\int_0^\infty \diffd u\, m''(s+u)\big(\cbes
\delta z {z+\mu}_u - y\big)}\Big] \frac{c|x-\mu|}{s+1}
e^{\frac{c|x-\mu|}{s+1}}\\
&= \mathcal E(\mu)\frac{c|x-\mu|}{s+1} e^{\frac{c|x-\mu|}{s+1}}.
\end{align}
Returning to \eqref{wtoE}, we obtain
\begin{align}
|w(y,s+\delta)-w(y,s)| &\le \int_{-\infty}^\infty \mathcal
E(\mu)\frac{c|x-\mu|}{s+1} e^{\frac{c|x-\mu|}{s+1}}
|x-\mu|\P(X\in\diffd x)\\
&= \mathcal E(\mu)\E\bigg[\frac{c(X-\mu)^2}{s+1} e^{\frac{c|X -\mu|}{s+1}}\bigg].
\end{align}
Finally, by scaling, conditionally on $\bes y_s=z$ we have
\begin{equation}
|X-\mu| \overset{(d)}{=} \sqrt\delta \Big|\bes {z/\sqrt\delta}_1 -\E\Big[\bes
{z/\sqrt\delta}_1 \Big]\Big|
\le \sqrt\delta\,\Gamma,
\end{equation}
where $\Gamma$ was defined in \eqref{gamdef} and is a non-negative random variable with Gaussian tail. Therefore
\begin{equation}
\E\bigg[\frac{c|X-\mu|^2}{s+1} e^{\frac{c|X -\mu|}{s+1}}\bigg] \le
C\frac{\delta}{s+1}
\end{equation}
for some constant $C$ provided $\delta \le s^2$, and one may check
similarly to \eqref{eJbdd} that $\mathcal E(\mu)$ is also bounded uniformly
in $y$, $s$ and $\delta$.
This establishes the last line of \eqref{interdeco} and
completes the proof. \end{proof}

\subsection{Proof of Proposition~\ref{finer}}\label{prooffiner}
To prove Proposition \ref{finer} we proceed via three lemmas. We first write $I_t(y)=I(y)-\big(I(y)-I_t(y)\big)$, and show that the correction $I(y)-I_t(y)$ is small in the
following sense:
\begin{lem}\label{lemIIC} Suppose that $m$ is twice continuously
differentiable and satisfies \eqref{newmprop}. Then there exist positive random variables $G$ and $G_t$ with Gaussian
tails, where all the $G_t$ have the same distribution, such
that uniformly in $y$,
\begin{equation}
I(y) = G \mathcal O(1) \qquad \hbox{and} \qquad I(y)-I_t(y) = G_t
\mathcal O\big({t^{-\frac{1}2}}\big).
\end{equation}
\end{lem}
Unsurprisingly, for random variables with Gaussian tails we can make
series expansions rather easily:
\begin{lem}\label{lemdl}
Let $G$ and $G_t$ be positive random variables with Gaussian
tails such that  all the $G_t$ have the same distribution.
Suppose that $A_t$ and $B_t$ are random variables such that
\begin{equation}
A_t= G\mathcal O(1),\qquad
B_t= G_t\mathcal O( \epsilon_t)
\end{equation}
where $\epsilon_t\ge 0$ is a deterministic function with $\epsilon_t\to0$ as $t\to\infty$. 
Then for any integer $n\ge0$,
\begin{equation}
\E\big[e^{A_t+B_t}\big]=\sum_{p=0}^{n} \frac1{p!}\E\big[e^{A_t}B_t^p\big]
+\mathcal O(\epsilon_t^{n+1}).
\end{equation}
\end{lem}
Taking $n=1$, $\epsilon_t
=t^{-1/2}$, $A_t=I(y)$ and $B_t=-\big(I(y)-I_t(y)\big)$, we find
\begin{equation}
\E[e^{I_t(y)}] = \E[e^{I(y)}] - \E\Big[e^{I(y)}\big(I(y)-I_t(y)\big)\Big] +\mathcal
O\Big(\frac1t\Big).
\end{equation}
The difficult part is then to show how the $I(y)$ decorrelates asymptotically from $I(y)-I_t(y)$:
\begin{lem}
\label{maindeco}
Suppose that $m$ is twice continuously differentiable with $m''(t)
= \frac{3}{2(t+1)^2} + r''(t)$ where $r(t)=\mathcal O(t^{-2-\eta})$ for some $\eta>0$. Then
\begin{equation}
 \E\Big[e^{I(y)}       \big(I(y)-I_t(y)\big)\Big] 
=\E\Big[e^{I(y)}\Big]\E\big[I(y)-I_t(y)\big]
+ \mathcal O\Big(\frac{\log t}t\Big) + y\mathcal O\Big(\frac{1}{t}\Big).
\end{equation}
\end{lem}

Of course~$\psi_\infty(y) = \E[e^{I(y)}]$ and $\psi_t(y) = \E[e^{I_t(y)}]$,
so these lemmas together give Proposition~\ref{finer}. It remains to prove
the lemmas.

\begin{proof}[Proof of Lemma~\ref{lemIIC}]
The bound on $I(y)$ is easy by applying Lemma~\ref{bessfluc} since
$|m''(s)|\le \frac{c}{(1+s)^2}$ for all $s$ and some constant $c$. We now
turn to $I_t(y)-I(y)$.

Recall the expression~\eqref{Iagain} of $I_t(y)$, replace $m''(s)$ by its expression \eqref{newmprop} and cut the integral into three pieces to obtain
\begin{align}
I_t(y)& = \frac12 \int_0^t \diffd s\,\,
m''(s)\frac{t-s}{t}\Big(\bes{y}_{\frac{st}{t-s}}-y\Big),\\ &=
\frac{3}{4}\,\frac{t+1}{t}\int_0^t \frac{\diffd
s}{(s+1)^2}\Big(\bes{y}_{\frac{st}{t-s}}-y\Big) - \frac{3}{4t}\int_0^t
\frac{\diffd s}{s+1}\Big(\bes{y}_{\frac{st}{t-s}}-y\Big)+
\frac12\int_0^t\diffd s\, r''(s) \frac{t-s}t \Big( \bes y_{ \frac{st}{ t-s}}-y\Big).\label{cutIt}
\end{align}
Recall that, by scaling,
\begin{equation}
\bes y_{tu} = \sqrt t\,\tbes {\tilde y}_u \quad\text{ with } \quad \tilde
y=y/\sqrt t
\label{scalebes1}
\end{equation}
where $\tbes{\tilde y}_u$ is another, $t$ dependent (implicit in notation),
Bessel process started from $\tilde y$.
We can apply Lemma~\ref{bessfluc} to the Bessel process
$\tbes{\tilde y}_u$ but, as it depends on $t$, the random variable $G$ must
be replaced by some other random variable $\tilde G_t$ which has the same
Gaussian tails as $G$.
Then
\begin{equation}
\big|\tbes {\tilde y}_{u}- {\tilde y}\big|\le \tilde G_t
\max\Big(u^{\frac12-\epsilon},u^{\frac12+\epsilon}\Big)
\quad\text{so}
\quad
\big|\bes {y}_{tu}- {y}\big|\le \tilde G_t \sqrt t
\max\Big(u^{\frac12-\epsilon},u^{\frac12+\epsilon}\Big).
\label{scaledlem}
\end{equation}

In the second integral of \eqref{cutIt}, make the change of variable
$u=s/t$ and use~\eqref{scaledlem} to obtain
\begin{equation}
\left|\frac{1}{t} \int_{0}^t\frac{\diffd s}{s+1} \Big[ \bes y_{ \frac{st}{
t-s}}-y\Big]\right|\le \frac{1}t\int_{0}^1\frac{\diffd u}{u} \tilde
G_t \sqrt t
\max
\bigg\{
\Big(\frac u{1-u}\Big)^{\frac12+\epsilon}
,
\Big(\frac u{1-u}\Big)^{\frac12-\epsilon}
\bigg\}
=\tilde G_t \mathcal O \big(t^{-\frac12}\big).
\label{cutIt2}
\end{equation}

In the first integral of~\eqref{cutIt},
make the change of variable $u=st/(t-s)$ to obtain
\begin{equation}
I_t(y)=
\frac34\,\frac{t+1}{t} \int_{0}^\infty\frac{\diffd u}{(u+1+u/t)^2}\big( \bes y_u-y\big)
+\frac12\int_0^t\diffd s\, r''(s) \frac{t-s}t\Big( \bes y_{ \frac{st}{ t-s}}-y\Big)
+\tilde G_t \mathcal O \big(t^{-\frac12}\big).
\label{newIt}
\end{equation}

We now turn to $I(y)$. In expression~\eqref{Iagain} of $I(y)$, use the expression \eqref{newmprop} and cut the integral into the following pieces:
\begin{multline}\label{newI}
I(y)
=
\frac{3}{4} \int_0^\infty\frac{\diffd s}{(s+1)^2}\big( \bes y_s-y\big)
+\frac12\int_0^t\diffd s\, r''(s) \frac{t-s}t\big( \bes y_s-y\big)
\\
+\frac12\int_0^t\diffd s\, r''(s) \frac st\big( \bes y_s-y\big)
+\frac12\int_t^\infty\diffd s\, r''(s) \big( \bes y_s-y\big).
\end{multline}
Applying Lemma \ref{bessfluc} and the fact that $r''(s)$ is bounded (since it is continuous on $[0,\infty)$ and tends to $0$) with
$r''(s)=\mathcal O(s^{-2-\eta})$ for some $\eta>0$, it is easy to check that
the third and fourth integrals are bounded in modulus by $G \mathcal
O(t^{-1/2})$ if $\epsilon<\eta$. Using Lemma \ref{bessfluc} again, it is
also easy to check that the first terms in \eqref{newIt} and \eqref{newI}
are equal up to an error of size $G\mathcal O(1/t)$ which we absorb in the
$G\mathcal O(t^{-1/2})$ that we already have. Thus we get
\begin{equation}
I_t(y)-I(y) = \frac12\int_0^t\diffd s \, r''(s) \frac{t-s}t
\Big(\bes y_{ \frac{st}{ t-s}}-\bes y_s\Big)
+\tilde G_t \mathcal O\big(t^{-\frac12}\big)
+G \mathcal O\big(t^{-\frac12}\big).
\end{equation}
We now focus on the remaining integral. The difference
$\bes y_{ st/( t-s)}-\bes y_s$ is the position at time
$s^2/(t-s)=st/(t-s)-s$ of a new
Bessel process started from $\bes y_s$. It is also, by scaling, equal to
$t^{-1/2}$ times the position at time $ts^2/(t-s)$ of another Bessel process
started from $\sqrt t \bes y_s$. Applying Lemma~\ref{bessfluc} again to
this last Bessel process, we get
\begin{equation}
\Big|\bes y_{ \frac{st}{ t-s}}-\bes y_s\Big| \le \frac{\hat G_t}{\sqrt
t}
\max\bigg\{
\Big(\frac{t s^2}{ t-s}\Big)^{\frac12+\epsilon},
\Big(\frac{t s^2}{ t-s}\Big)^{\frac12-\epsilon}
\bigg\}
\le\frac{\hat G_t}{\sqrt t}\times
\begin{cases}
\displaystyle \frac t{t-s}s^{1+2\epsilon}&\text{if $1<s<t$},\\[2ex]
\displaystyle \frac t{t-1}&\text{if $0<s<1$}.
\end{cases}
\end{equation}
where $\hat G_t$ is another $t$-dependent positive random variable with the
same Gaussian tail as $G$.
Since $r''(s)$ is bounded and $r''(s)=\mathcal O(s^{-2-\eta})$,
the integral
$\int_1^\infty\diffd s\, r''(s)s^{1+2\epsilon}$ is finite provided
$\epsilon<\eta/2$, and we obtain
\begin{equation}
I_t(y)-I(y) = 
G_t \mathcal O\big(t^{-\frac12}\big),
\end{equation}
with $G_t=\max(G,\tilde G_t,\hat G_t)$. This concludes the proof.
\end{proof}

\begin{proof}[Proof of Lemma~\ref{lemdl}]
With the hypothesis of the lemma, write $|A_t|\le \alpha G$ and $|B_t|\le
\beta\epsilon_t G_t$ for some $\alpha>0$ and $\beta>0$. Writing
\begin{equation}
e^{A_t+B_t}=\sum_{p=0}^\infty \frac1{p!}e^{A_t}B_t^p,
\end{equation}
we can apply dominated convergence---since the partial sums are dominated by $\exp(A_t+|B_t|)$ which has finite expectation---and obtain
\begin{equation}
\E\big[e^{A_t+B_t}\big]=\sum_{p=0}^\infty
\frac1{p!}\E\big[e^{A_t}B_t^p\big].
\end{equation}
It only remains to show that the sum for $p\ge n+1$ is $\mathcal
O(\epsilon_t^{n+1})$. To do this observe that
\begin{equation}
\Big|\frac1{p!}\E\big[e^{A_t}B_t^p\big]\Big|\le
\frac{\epsilon_t^p}{p!}\E\big[e^{\alpha G}(\beta G_t)^p\big]
\le
\epsilon_t^p\E\big[e^{\alpha G+\beta G_t}\big],
\end{equation}
where the last expectation is finite. Then,
as soon as $\epsilon_t<1$, we have
\begin{equation}
\Big|\sum_{p=n+1}^\infty \frac1{p!}\E\big[e^{A_t}B_t^p\big]\Big|
\le \frac{\epsilon_t^{n+1}}{1-\epsilon_t}\E\big[e^{\alpha G+\beta
G_t}\big],
\end{equation}
which concludes the proof.
\end{proof}

\begin{proof}[Proof of Lemma~\ref{maindeco}]
Define
\begin{equation}
J_t(y)
= 2\E\Big[e^{I(y)}\big(I(y)-I_t(y)\big)\Big]
-2\E\Big[e^{I(y)}\Big]\E\big[I(y)-I_t(y)\big].
\end{equation}
We want to show that $J_t(y) = \mathcal O\big(\frac{\log t}{t}\big)+y\mathcal
O\big(\frac1t\big)$. Clearly,
\begin{align}
J_t(y) &= 2\Big(\E[e^{I(y)}I(y)] - \E[e^{I(y)}]\E[I(y)]\Big) 
- 2\Big(\E[e^{ I(y)}I_t(y)] - \E[e^{ I(y)}]\E[I_t(y)]\Big)\\
&= \int_0^\infty\diffd s\,m''(s) \Big(\E\Big[e^{I(y)}\big(\bes
y _s-y\big)\Big]- \E\Big[e^{I(y)}\Big]\E\Big[\bes y _s-y\Big]\Big)\notag\\
&\hspace{20mm} - \int_0^t\diffd s\,m''(s)
\frac{t-s}t\Big(\E\Big[e^{I(y)}\Big(\bes
y _{\frac{ts}{t-s}}-y\Big)\Big]-\E\Big[e^{I(y)}\Big]\E\Big[\bes
y _{\frac{ts}{t-s}}-y\Big]\Big)\\
&= \int_0^\infty\diffd s\,m''(s) w(y,s) - \int_0^t\diffd s\,m''(s)
\frac{t-s}{t}w\Big(y,\frac{ts}{t-s}\Big),
\end{align}
where we recall the definition of $w$ from \eqref{wdef}. We now apply
Proposition \ref{wlem}.
Cut the integrals at $t/2$ and rearrange the terms:
\begin{multline}
J_t(y) =\int_{\frac t2}^{\infty}\diffd s\, m''(s)w(y,s)
+\int_0^{\frac t2}\diffd s\,m''(s)\frac st  w(y,s)
-\int_{\frac t2}^t \diffd s\, m''(s) \frac{t-s}tw\Big(y,\frac{ts}{t-s}\Big)
\\
 -\int_0^{\frac t2}\diffd s\, m''(s) \frac{t-s}{t} \Big(
w\Big(y,\frac{ts}{t-s}\Big)-w(y,s)\Big).
\end{multline}
Using from Proposition \ref{wlem} that $|w(y,s)|\le C\log(s+1)$ and of course $m''(s)=\mathcal O(1/s^2)$, the first
and third integrals are both $\mathcal O(\frac{\log t}{t})$, uniformly in $y$.
Now using from Proposition \ref{wlem} that $|w(y,s)|\le
C\big(1+y\frac{\log(s+1)}{\sqrt s}\big)$, the second integral is $y\mathcal
O\big(\frac1t\big) + \mathcal O(\frac{\log t}{t})$.

We now turn to the fourth integral.
Writing $\frac{st}{t-s}=s+\frac{s^2}{t-s}$ and noticing that for $s<t/2$
we have $\frac{s^2}{t-s}<s^2$ as soon as $t\ge2$, the last part of Proposition \ref{wlem} gives
$\big|w\big(y,\frac{ts}{t-s}\big)-w(y,s)\big|\le C\frac
s{t-s}$, and therefore the fourth integral is $\mathcal O\big(\frac{\log
t}t\big)$, which concludes the proof.
\end{proof}

\subsection{Proof of Proposition~\ref{finer2}}\label{prooffiner2}

For $y\ge0$ we introduce the notation $\mu(y,t) := \E[\bes y_t]-y$ and observe that 
\begin{equation} \label{prop of mu}
\mu(y,tu) = \sqrt{t} \mu(\frac{y}{\sqrt{t}} ,u)
, \quad
\mu(0,s) = \frac4{\sqrt\pi} \sqrt{s}
, \quad
\max\big[0,\mu(0,s)-y\Big] \le\mu(y,s) \le \mu(0,s)
.
\end{equation}
(The first equality is the scaling property, and the inequalities are from \eqref{beslem1}. The second equality can be calculated directly from the probability density function for a Bessel process; see for example \cite[page 446]{Revuz1999}.)

With this notation we can rewrite 
\begin{align} 
\label{I with mu}
\E[I(y)] &= \frac12\int_0^\infty \diffd s\, m''(s) \mu(y, s)  \\
\label{I_t with mu}
\E[I_t(y)] &= \frac12\int_0^t \diffd s\, m''(s) \frac{t-s}{t}
	\mu\Big(y, \frac{st}{ t-s}\Big). 
\end{align}

As usual we use the expression \eqref{newmprop}, decomposing $\E\big[I(y)-I_t(y)\big]$ into terms containing $3/2(s+1)^2$ and terms containing $r''(s)$. In the former we make our usual change of time $u = st/(t-s)$, but in the latter we do not.
\begin{multline}
\E\big[I(y)-I_t(y)\big] = \frac34 \int_0^\infty \diffd s\,\frac{1}{(s+1)^2}\mu (y,s) - \frac34 \int_0^\infty \diffd u\,\frac{1}{(\frac{tu}{t+u}+1)^2}\Big(\frac{t}{t+u}\Big)^3 \mu(y,u)\\
+\frac12 \int_0^\infty \diffd s\, r''(s) \mu(y,s) - \frac12 \int_0^t \diffd s\,r''(s)\,\frac{t-s}{t}\,\mu\Big(y,\frac{st}{t-s}\Big).
\end{multline}
Rearranging we get
\begin{equation}\label{decomposition de I-I_t}
\begin{aligned}
\E\big[I(y)-I_t(y)\big] = &\frac 34\int_0^\infty \diffd s \,\left(1 -\frac{t}{t+s} \right)\frac{1}{(s+1)^2}\,\mu(y, s)\\
&+\frac 34\int_0^\infty \diffd s \,\left(\frac{1}{(s+1)^2} - \frac{1}{(s+1+s/t)^2}\right)\frac{t}{t+s}\,\mu(y,s)\\
&+ \frac12\int_0^t \diffd s  \, r''(s) \left(   \mu(y, s) - \frac {t-s}t \,\mu\Big(y, \frac{st}{t-s}\Big)  \right) +\frac12\int_t^\infty\diffd s\, r''(s) \mu(y, s),
\end{aligned}
\end{equation}
and we treat each of the four integrals on the right-hand side in turn.

\subsubsection*{The first integral in the right hand side of \eqref{decomposition de I-I_t}}

Making the change of variable $s=tu$ and using the first part of \eqref{prop of mu} we have
\begin{equation}
\int_0^\infty \diffd s \,\left(1 -\frac{t}{t+s} \right)\frac{1}{(s+1)^2}\,\mu(y, s) =  \frac1{\sqrt t} \int_{0}^\infty \diffd u\, \frac{u}{(u+1)(u+1/t)^2}  \mu\Big(\frac y{\sqrt{t}} ,u\Big).
\end{equation}
We now approximate $\mu(y/\sqrt t,u)$ by $\mu(0,u)$, bounding the error by using the last part of~\eqref{prop of mu}:
\begin{multline}
\left| \frac1{\sqrt t} \int_{0}^\infty \diffd u\, \frac{u}{(u+1)(u+1/t)^2}  \mu\Big(\frac y{\sqrt{t}} ,u\Big) -  \frac1{\sqrt t} \int_{0}^\infty \diffd u\, \frac{u}{(u+1)(u+1/t)^2}  \mu(0 ,u)\right|\\
\le  \frac1{\sqrt t} \int_{0}^\infty \diffd u\, \frac{u}{(u+1)(u+1/t)^2} \frac y{\sqrt{t}}.
\end{multline}
The right hand side is $y\mathcal O\big(\frac{\log t}{t}\big)$, and using the second part of \eqref{prop of mu}, we have
\begin{equation}
\frac1{\sqrt t} \int_{0}^\infty \diffd u\, \frac{u}{(u+1)(u+1/t)^2}  \mu(0 ,u) = \frac{4\sqrt\pi}{\sqrt t} + \mathcal O(1/t).
\end{equation}
We therefore conclude that
\begin{equation}
\int_0^\infty \diffd s \,\left(1 -\frac{t}{t+s} \right)\frac{1}{(s+1)^2}\,\mu(y, s) = \frac{4\sqrt\pi}{\sqrt t} + y\mathcal O\Big(\frac{\log t}{t}\Big) + \mathcal O\Big(\frac{1}{t}\Big).
\end{equation}

\subsubsection*{The second integral in the right hand side of \eqref{decomposition de I-I_t}}

We note that
\begin{equation}
\frac{1}{(s+1)^2} - \frac{1}{(s+1+s/t)^2} = \frac{1}{(s+1)^2}\mathcal O\Big(\frac1t\Big),
\end{equation}
and $t/(t+s)\le 1$, so using the bound $\mu(y,s)\le \mu(y,0) = 4\sqrt s/\sqrt \pi$ from \eqref{prop of mu}, we easily see that the second integral is $\mathcal O(1/t)$ uniformly in $y$.

\subsubsection*{The third integral in the right hand side of \eqref{decomposition de I-I_t}}
We use the following result: for any $\delta>0$,
\begin{equation}
0\le\mu(y,s+\delta)-\mu(y,s) \le \mu(0,s+\delta)-\mu(0,s)
= \frac4{\sqrt\pi}\Big(\sqrt{s+\delta}-\sqrt s\Big).
\end{equation}
This follows from the Markov property plus \eqref{prop of mu}. Then
\begin{equation}
\frac {t-s}t \mu\Big(y, \frac{st}{t-s}\Big) -\mu(y,s)
=\frac {t-s}t \Big[\mu\Big(y, \frac{st}{t-s}\Big)-\mu(y,s) \Big]-\frac st
\mu(y,s),
\end{equation}
so that
\begin{equation}
-\frac4{\sqrt\pi}\frac{s^{3/2}}t
\le\frac {t-s}t \mu\Big(y, \frac{st}{t-s}\Big) -\mu(y,s)\le
\frac4{\sqrt\pi}\frac{t-s}t\Big[\Big(\frac{st}{t-s}\Big)^{1/2}- s^{1/2}\Big]
\end{equation}
But $(\frac{st}{t-s})^{1/2}= s^{1/2} (1+\frac s{t-s})^{1/2} \le s^{1/2}\big(1+\frac
s{2(t-s)}\big)$ so the right hand side of the previous equation is at most $(4/\sqrt\pi)\times s^{3/2}/(2t)$. We conclude that
\begin{equation}
\left|\int_0^t \diffd s  \, r''(s) \left[   \mu(y, s)
- \frac {t-s}t \mu\left(y, \frac{st}{t-s}\right)  \right]\right|\le 
\frac4{\sqrt\pi t}
\int_0^t\diffd s\, s^{3/2}\big|r''(s)\big|
=\begin{cases}
\displaystyle\mathcal O\left(\frac1t\right) &\text{if $\eta>\frac12$},\\[2ex]
\displaystyle\mathcal O\left(\frac{\log t}t\right) &\text{if
$\eta=\frac12$},\\[2ex]
\displaystyle\mathcal O\left(\frac1{t^{\frac12+\eta}}\right) &\text{if
$\eta<\frac12$},
\end{cases}
\end{equation}
uniformly in $y$.

\subsubsection*{The fourth integral in the right hand side of \eqref{decomposition de I-I_t}}

Since $r''(s) = O(s^{-2-\eta})$ for some $\eta>0$, using \eqref{prop of mu} again we have
\begin{equation}
\left|\int_t^\infty\diffd s\, r''(s) \mu(y, s)\right|
\le \frac4{\sqrt\pi}\int_t^{\infty}\diffd s\,|r''(s)|\sqrt s = 
 \mathcal O\left(\frac1{t^{\frac12+\eta}}\right)
\end{equation}
uniformly in $y$.

\bigskip

Putting together the results from the four integrals give the proposition.

\section*{Appendix}

\begin{lem}\label{bessfluceasy}
For any $\epsilon>0$, the non-negative random variable
\begin{equation}
G := \sup_{s>0} \frac{\bes 0_{s}}{\max\big(s^{1/2-\epsilon},
s^{1/2+\epsilon}\big)}
\label{gausstails}
\end{equation}
has Gaussian tail under $\P$.
\end{lem}

\begin{proof}
We do this in two parts, first considering the supremum over 
$s\in(0,1]$. We have
\begin{equation}
\P\bigg(\sup_{s\in(0,1]} \frac{\bes 0_s}{s^{1/2-\epsilon}} > z\bigg) 
\leq \sum_{n=2}^\infty 
\P\bigg(\sup_{s\in\big(\tfrac{1}{n},\tfrac{1}{n-1}\big]}
\frac{\bes 0_s}{s^{1/2-\epsilon}} 
 > z\bigg).
\end{equation}
By scaling, this equals
\begin{equation}
\sum_{n=2}^\infty\P\bigg(\sup_{s\in\big(1,\tfrac{n}{n-1}\big]} 
\frac{\bes 0_s}{s^{1/2-\epsilon}} > zn^\epsilon\bigg) \leq 
\sum_{n=2}^\infty\P\bigg(\sup_{s\in(1,2]} \bes 0_s > zn^\epsilon\bigg).
\end{equation}
Now note that there exist $c_3>0$ and $c_4>0$
such that $\P\big(\sup_{s\in(1,2]} \bes 0_s > z\big) \leq c_3 \exp\big[{-c_4
z^2}\big]$ for all $z>0$, so
\begin{equation}
\P\bigg(\sup_{s\in(0,1]} \frac{\bes 0_s}{s^{1/2-\epsilon}} > z\bigg) 
\leq c_3 \sum_{n=2}^\infty e^{-c_4 z^2 n^{2\epsilon} },
\end{equation}
and it is an easy exercise to show that there exist $c_1$ and $c_2$ (with
$c_1$ depending on $\epsilon$) such that $c_3\sum_{n=2}^\infty e^{-c_4
z^2n^{2\epsilon}} \leq c_1 e^{-c_2 z^2}$.

Similarly for $s\in(1,\infty)$,
\begin{equation}
\P\bigg(\sup_{s\in(1,\infty)} \frac{\bes 0_s}{s^{1/2+\epsilon}} > 
z\bigg) \leq \sum_{n=1}^\infty 
\P\bigg(\sup_{s\in(n,n+1]}\frac{\bes 0_s}{s^{1/2+\epsilon}} > z\bigg).
\end{equation}
By scaling, this equals
\begin{equation}
\sum_{n=1}^\infty\P\bigg(\sup_{s\in(1,\tfrac{n+1}{n}]} 
\frac{\bes 0_s}{s^{1/2+\epsilon}} > zn^\epsilon\bigg) \leq 
\sum_{n=1}^\infty\P\bigg(\sup_{s\in(1,2]} \bes 0_s > zn^\epsilon\bigg).
\end{equation}
and the end of the argument is the same as in the previous case.
\end{proof}

\providecommand{\bysame}{\leavevmode\hbox to3em{\hrulefill}\thinspace}
\providecommand{\MR}{\relax\ifhmode\unskip\space\fi MR }
\providecommand{\MRhref}[2]{%
  \href{http://www.ams.org/mathscinet-getitem?mr=#1}{#2}
}
\providecommand{\href}[2]{#2}


\begin{thebibliography}{KPP37}

\bibitem[AW78]{Aronson1978}
Donald~G Aronson and Hans~F Weinberger, \emph{Multidimensional nonlinear
  diffusion arising in population genetics}, Advances in Mathematics
  \textbf{30} (1978), no.~1, 33--76.

\bibitem[BD15]{Brunet2015}
{\'E}ric Brunet and Bernard Derrida, \emph{An exactly solvable travelling wave
  equation in the {F}isher--{KPP} class}, Journal of Statistical Physics
  (2015), 1--20.

\bibitem[Bra83]{Bramson83}
Maury Bramson, \emph{Convergence of solutions of the {K}olmogorov equation to
  travelling waves}, Mem. Amer. Math. Soc. \textbf{44} (1983), no.~285, iv+190.

\bibitem[EvS00]{Ebert2000}
Ute Ebert and Wim van Saarloos, \emph{Front propagation into unstable states:
  universal algebraic convergence towards uniformly translating pulled fronts},
  Physica D: Nonlinear Phenomena \textbf{146} (2000), no.~1, 1--99.

\bibitem[Fis37]{Fisher1937}
R.~A. Fisher, \emph{The advance of advantageous genes}, Ann. Eugenics
  \textbf{7} (1937), 355--369.

\bibitem[HE15]{HerrmannTanre2015}
Samuel Herrmann and Tanr\'e; Etienne, \emph{The first-passage time of the
  brownian motion to a curved boundary: an algorithmic approach}, arXiv
  preprint arXiv:1501.07060 (2015).

\bibitem[Hen14]{Henderson2014}
Christopher Henderson, \emph{Population stabilization in branching {B}rownian
  motion with absorption}, arXiv preprint arXiv:1409.4836 (2014).

\bibitem[KPP37]{Kolmogorov1937}
A.~N. Kolmogorov, I.~Petrovski, and N.~Piscounov, \emph{\'{E}tude de
  l'\'{e}quation de la diffusion avec croissance de la quantit\'{e} de
  mati\`{e}re et son application \`{a} un problem biologique}, Mosc. Univ.
  Bull. Math. \textbf{1} (1937), 1--25, Translated and reprinted in Pelce, P.,
  \emph{Dynamics of Curved Fronts} (Academic, San Diego, 1988).

\bibitem[Kun97]{kunita1997}
Hiroshi Kunita, \emph{Stochastic flows and stochastic differential equations},
  vol.~24, Cambridge university press, 1997.

\bibitem[MM14]{Mueller2014}
AH~Mueller and S~Munier, \emph{Phenomenological picture of fluctuations in
  branching random walks}, Physical Review E \textbf{90} (2014), no.~4, 042143.

\bibitem[RY99]{Revuz1999}
Daniel Revuz and Marc Yor, \emph{Continuous martingales and brownian motion},
  vol. 293, Springer Science \& Business Media, 1999.

\end{thebibliography}
\end{document}